\documentclass{amsart}
\usepackage[T1]{fontenc}
\usepackage{amssymb}
\usepackage{amsmath}
\usepackage{amsthm}
\usepackage{enumerate}
\usepackage[hyphens]{url}
\usepackage{hyperref}
\usepackage{tikz}
\usepackage{tikz-cd}
\hypersetup{
    colorlinks=true,
    linkcolor=blue,
}

\usepackage[
    backend=biber,
    style=alphabetic,
    sorting=nyt,
    useprefix=true,
    doi=false,
    maxnames=99
]{biblatex}
\newtheorem{theorem}{Theorem}[section]
\newtheorem{lemma}[theorem]{Lemma}
\newtheorem{proposition}[theorem]{Proposition}
\newtheorem{corollary}[theorem]{Corollary}

\theoremstyle{definition}
\newtheorem{definition}[theorem]{Definition}
\newtheorem{remark}[theorem]{Remark}

\newcommand{\E}{\mathbb{E}}
\newcommand{\N}{\mathbb{N}}

\newcommand{\Z}{\mathbb{Z}}

\renewcommand{\O}{\mathcal{O}}
\newcommand{\R}{\mathcal{R}}

\newcommand{\I}{\mathrm{I}}
\renewcommand{\L}{\mathrm{L}}
\newcommand{\op}{\mathrm{op}}
\renewcommand{\d}{\mathrm{d}}

\newcommand{\Ani}{\mathsf{Ani}}
\newcommand{\Fun}{\mathsf{Fun}}

\newcommand{\D}{\mathsf{D}}
\newcommand{\Ring}{\mathsf{Ring}}
\newcommand{\Perf}{\mathsf{Perf}}
\newcommand{\cProj}{\mathsf{cProj}}
\newcommand{\Alg}{\mathsf{Alg}}
\newcommand{\fPoly}{\mathsf{fPoly}}

\newcommand{\sslash}{/\!\!/}

\DeclareMathOperator{\id}{id}
\DeclareMathOperator{\Map}{Map}
\DeclareMathOperator{\End}{End}
\DeclareMathOperator{\map}{map}
\DeclareMathOperator{\en}{end}
\DeclareMathOperator{\Hom}{Hom}
\DeclareMathOperator{\LSym}{LSym}
\DeclareMathOperator{\fib}{fib}
\DeclareMathOperator{\cofib}{cofib}
\DeclareMathOperator{\K}{K}

\renewcommand{\th}[1]{#1\textsuperscript{th}}

\DeclareMathOperator*{\colim}{colim}

\begin{document}

\title{Slicing criterion for ind-smooth ring maps}
\author{Longke Tang}

\begin{abstract}
    We show that ind-smoothness of flat ring maps can be tested on constructible stratifications, even for maps of non-Noetherian rings. We prove this by generalizing ind-smoothness and ind-lci to a sequence of conditions on animated ring maps called ind-$d$-smoothness, and showing all at once that they can be tested on constructible stratifications. 
\end{abstract}

\maketitle

\section{Introduction}

\subsection{Results}
The Artin approximation theorem \cite[Theorem 1.10]{artin-approximation} is a very useful technical result in commutative algebra with many important consequences, including Artin's representability theorem and the proper base change theorem for \'etale cohomology. The most memorable generalization of the Artin approximation theorem, proved by Popescu, is as follows: 

\begin{theorem}[{\cite[Theorem 1.8]{popescu-approximation}}; {\cite[\texttt{07GC}]{stacks}}]\label{popescu}
    A map $R\to A$ of Noetherian rings is ind-smooth if and only if it is regular. 
\end{theorem}

Here, $R\to A$ being \emph{ind-smooth} means that there exists a filtered system $(A_i)_{i\in I}$ of smooth $R$-algebras such that $A=\colim_{i\in I}A_i$.

While the statement of this theorem is clean, its proof is extremely technical. An excellent exposition of the proof is \cite[\texttt{07BW}]{stacks}, where one first reduces to the case where $R$ is a field and then treats it by very concrete constructions. 

Recent developments in $p$-adic geometry have led us to consider non-Noetherian rings such as perfectoid rings and general valuation rings. For example, Bouis \cite{bouis-smooth} has proved a comparison theorem between the \'etale cohomology and the syntomic cohomology of certain large algebras over perfectoid rings, by showing that valuation rings over perfectoid rings have a regularity property called $F$-smoothness, introduced in \cite[Definition 4.1]{bhatt-mathew-syntomic}. Therefore, it is natural to consider possible non-Noetherian generalizations of Theorem \ref{popescu}. Unfortunately, we currently have no idea what a non-Noetherian version of Theorem \ref{popescu} will be, if there is one. Nevertheless, we are able to prove the following theorem, which is an analog of the reduction step carried out in \cite[\texttt{07F1}]{stacks}: 

\begin{theorem}[cf.\ {\cite[\texttt{07F5}]{stacks}}]\label{main-thm-classical}
    Let $R$ be a ring, $\Lambda$ be a flat $R$-algebra, and $r$ be an element of $R$. If $\Lambda[1/r]$ is ind-smooth over $R[1/r]$ and $\Lambda/r\Lambda$ is ind-smooth over $R/r$, then $\Lambda$ is ind-smooth over $R$. 
\end{theorem}

\begin{corollary}
    Let $R$ be a ring and $\Lambda$ be a flat $R$-algebra. Let $r_1,\ldots,r_n$ be elements of $R$ and set $r_{n+1}=1$. If $(\Lambda/(r_1,\ldots,r_i)\Lambda)[1/r_{i+1}]$ is ind-smooth over $(R/(r_1,\ldots,r_i))[1/r_{i+1}]$ for $i=0,1,\ldots,n$, then $\Lambda$ is ind-smooth over $R$. 
\end{corollary}

The following consequence of Theorem \ref{main-thm-classical} seems unknown in the literature, although Gabber, in his talk \cite{gabber-talk} on \cite{gabber-zavyalov}, has announced the special case of Theorem \ref{main-thm-classical} with $R/r=\Lambda/r\Lambda$ that suffices to imply it: 

\begin{corollary}\label{corollary-cp}
    Let $\O$ be a valuation ring of rank $1$ with fraction field of characteristic $0$, and let $\pi\in\O$ be a pseudo-uniformizer. Let $R$ be a finitely presented $\O$-algebra. Then the map from $R$ to its $\pi$-completion is ind-smooth. 
\end{corollary}

It is nontrivial to deduce Corollary \ref{corollary-cp} from our main theorems; this uses several standard facts on affinoid algebras. We will address it at the beginning of \S\ref{proof-main}. 

Our proof of Theorem \ref{main-thm-classical} uses in an essential way the theory of animated rings as is developed in \cite[\S 25]{sag} and is recalled and used in \cite[\S 5]{purity-flat}. Working with animated rings, one no longer worries about zerodivisors as one does in \cite[\texttt{07CN},\texttt{07CQ}]{stacks}; instead, one introduces higher homotopies when taking quotient by a zerodivisor so that the ``kernel'' stays free, and tries to control the resulting higher homotopies. Specifically, we generalize the definition of smoothness and the statement of Theorem \ref{main-thm-classical} as follows: 

\begin{definition}[Definition \ref{def-smooth}]
    Let $R$ be an animated ring and $d$ be a natural number. An animated $R$-algebra $A$ is \emph{$d$-smooth} if it is a compact object in the category of animated $R$-algebras, and its cotangent complex has tor-amplitude concentrated in $[0,d]$. It is \emph{ind-$d$-smooth} if it is a filtered colimit of $d$-smooth algebras. 
\end{definition}

\begin{remark}
    Here compactness is the animated analog of the classical notion of finite presentation. However, it behaves better than its classical counterpart in that it implies bounded tor-amplitude of the cotangent complex (Remark \ref{property-smooth}). This is a crucial point of this paper, and is the main reason why we need animated rings to prove Theorem \ref{main-thm-classical}, a purely classical statement. 
    
    Quillen showed in \cite[Theorem 5.5]{quillen-homology} that for maps of classical rings, $0$-smoothness is smoothness and $1$-smoothness is lci. He also conjectured in \cite[Conjectures 5.6, 5.7]{quillen-homology} that over a Noetherian classical ring, any $d$-smooth classical algebra is $2$-smooth, and is $1$-smooth if its tor-amplitude is finite. The latter conjecture has been solved by Avramov in \cite{avramov-lci}; see also \cite{briggs-iyengar-cotangent}. 
\end{remark}

\begin{theorem}\label{main-thm}
    Let $R$ be an animated ring and $S\subseteq R$ be a multiplicative submonoid. Let $d$ be a natural number and $\Lambda$ be an animated $R$-algebra. Then $\Lambda$ is ind-$d$-smooth over $R$ if: 
    \begin{itemize}
        \item $S^{-1}\Lambda$ is ind-$d$-smooth over $S^{-1}R$;
        \item for all $r\in S$, $\Lambda/r\Lambda$ is ind-$d$-smooth over $R/r$. 
    \end{itemize}
    Here, 
    \begin{itemize}
        \item a multiplicative submonoid $S\subseteq R$ means an $\E_\infty$-monoid $S$ along with a map to the multiplicative monoid of $R$ that induces an injection on $\pi_0$ and isomorphisms on higher homotopy groups; this agrees with the notion of subobjects defined in \S\ref{subsection-convention}, and such submonoids correspond one-to-one to multiplicative submonoids of $\pi_0(R)$. 
        \item all quotients are animated, namely the animated tensor product $-\otimes_{\Z[x]}\Z$ where $x\mapsto r$ on the left and $x\mapsto0$ on the right. 
    \end{itemize} 
\end{theorem}

Note that Theorem \ref{main-thm} does not immediately imply Theorem \ref{main-thm-classical}, as it is not clear that ind-smoothness of $R/r\to\Lambda/r\Lambda$ implies that of $R/r^n\to\Lambda/r^n\Lambda$, and the quotient in Theorem \ref{main-thm-classical} is not animated. We will address this implication in \S\ref{lifting-problem}. 

\begin{remark}[Animated vs simplicial]
    Before the adoption of animated rings, people used to use simplicial rings to describe and study derived phenomena in commutative ring theory, as in \cite[\texttt{08P5}]{stacks} for example. This paper chooses to use animated rings throughout. The author is not saying for certain that one cannot use simplicial rings to prove the main theorems in this paper; it is just that animated rings are much more convenient. For example, it would be very hard to rewrite Theorem \ref{factorize-smooth} in terms of simplicial rings, due to difficulties of simplicially presenting the animated quotient in Definition \ref{quotient-algebra} and doing enormously many cofibrant replacements throughout the proof. 
\end{remark}

\subsection{Conventions}\label{subsection-convention}
By ``rings'' we mean commutative rings. Everything is animated unless otherwise stated, that is: 
\begin{itemize}
    \item By ``categories'' we mean $\infty$-categories, unless we say ``ordinary categories''; 
    \item By ``rings'' we mean animated rings, unless we say ``classical rings''; 
    \item By ``modules'' we mean animated modules, unless we say ``classical modules'';
    \item \ldots
\end{itemize}
But by ``sets'' we still mean $0$-groupoids, because we have ``animas'' for $\infty$-groupoids. 

In any category, we say that a map $A\to B$ is an \emph{injection} or that $A$ is a \emph{subobject} of $B$, if the fiber product $A\times_BA$ exists and the diagonal $A\to A\times_BA$ is an isomorphism; \emph{surjections} refer to the class of maps left orthogonal to the class of injections, whenever this exists. 

We do not notationally distinguish between a ring or a module and its underlying anima, just as in classical commutative algebra we do not notationally distinguish between a ring or a module and its underlying set. 

We denote by $\Ani$ the category of animas, namely the category of spaces in \cite[Definition 1.2.16.1]{htt}, and by $\Fun$ the $(\infty,1)$-category of functors, as in \cite[Notation 1.2.7.2]{htt}. For a ring $R$ and $R$-algebras $A$ and $B$, we will consider their maps both as $R$-algebras and as $R$-modules; to avoid confusion, we denote by $\Hom_R(A,B)$ the anima of $R$-algebra maps, by $\map_R(A,B)$ the $R$-complex of $R$-module maps, and by $\Map_R(A,B)$ the $R$-module or the anima of $R$-module maps; we also denote $\en_R(A)=\map_R(A,A)$ and $\End_R(A)=\Map_R(A,A)$. 

\subsection{Acknowledgements}
I thank Ofer Gabber for bringing the problem into my view. I also want to thank Bhargav Bhatt, Kęstutis Česnavičius, Lars Hesselholt, and Bogdan Zavyalov for helpful discussions. Finally, I thank the anonymous referee for numerous suggestions improving the exposition of the paper. 

\section{Animated ring theory}

\subsection{Recollection on animated rings}
We first recall the theory of animated rings very concisely, and derive some useful facts. For a more detailed treatment of the basics of animated rings, see \cite[\S 5.1]{purity-flat} or \cite[\S 25]{sag}. 

\begin{definition}[Animated rings]
    Let $\fPoly$ denote the category of finitely generated polynomial $\Z$-algebras. \emph{Animated rings} are functors $\fPoly^\op\to\Ani$ preserving finite products. Denote by $\Ring$ the category of animated rings. The \emph{underlying anima} of an animated ring is its value on $\Z[x]$. 
\end{definition}

\begin{definition}[Modules]\label{definition-modules}
    The inclusion functor from $\fPoly$ to the category of connective $\E_\infty$-$\Z$-algebras left Kan extends to $\Ring$; we call the extended functor taking \emph{underlying $\E_\infty$-rings}. For a ring $R$, let $(\D(R),\otimes_R)$ denote the symmetric monoidal category of modules of its underlying $\E_\infty$-ring, and we call its objects \emph{$R$-complexes}; this comes equipped with a t-structure, and we call its connective objects \emph{$R$-modules}. Alternatively, one can animate the category of modules as in the end of \cite[\S 5.1.7]{purity-flat}, and define the category of complexes as the stabilization \cite[Definition 1.4.2.8]{ha} of the category of modules. 
\end{definition}

\begin{definition}[Derived symmetric powers]
    For a classical ring $R$ and a finite free $R$-module $M$, we define for $n\in\N$ the \emph{derived symmetric \th{$n$} power} $\LSym^n_R(M)$ as the classical symmetric \th{$n$} power of $M$ over $R$, and similarly the \emph{derived symmetric algebra} $\LSym_R(M)$ as the classical symmetric algebra of $M$ over $R$. 
    
    For a general ring $R$ and an $R$-module $M$, we define $\LSym^n_R(M)$ and $\LSym_R(M)$ by left Kan extension from the above case. It has many properties similar to the classical case, such as $\LSym_R(M)=\bigoplus_{n\in\N}\LSym^n_R(M)$, and if $M=R^{\oplus r}$ then $\LSym_R(M)=R[x_1,\ldots,x_r]$ is the polynomial algebra. For a more detailed treatment, see \cite[\S 25.2.2]{sag}. 
\end{definition}

\begin{definition}[Perfect complexes]
    For a ring $R$, let $\Perf(R)\subseteq\D(R)$ denote the full subcategory of compact objects in $\D(R)$, whose objects are called \emph{perfect $R$-complexes}; for $c,d\in\Z$, let $\Perf_{\ge c}(R)=\Perf(R)\cap\D_{\ge c}(R)$, 
    $$\Perf_{\le d}(R)=\{M\in\Perf(R)\mid\forall N\in\D_{>d}(R),\,\pi_0\Map_R(M,N)=0\},$$
    and $\Perf_{[c,d]}(R)=\Perf_{\ge c}(R)\cap\Perf_{\le d}(R)$, where $\D_{\ge c}$ means $c$-connective objects with respect to the t-structure in Definition \ref{definition-modules}; let $\cProj(R)=\Perf_{[0,0]}(R)$ and call its objects \emph{finite projective $R$-modules}. These subcategories constitute a bounded weight structure defined in \cite[Definition 2.2.1]{elmanto-sosnilo-weight}; in particular, 
    $$\Perf(R)=\bigcup_{c,d\in\Z}\Perf_{[c,d]}(R).$$
\end{definition}

The following proposition collects basic properties of $\Perf_{[c,d]}(R)$. 

\begin{proposition}\label{perfect-c-d}
    Let $R$ be a ring and $c$ and $d$ be integers. 
    \begin{enumerate}
        \item\label{one-step-free-resolution} For any $M\in\Perf_{\ge c}(R)$, there exists a finite free $R$-module $P$ and a map $P[c]\to M$ with cofiber $C\in\Perf_{\ge c+1}(R)$. If $d>c$ and $M\in\Perf_{[c,d]}(R)$, then in fact $C\in\Perf_{[c+1,d]}(R)$. 
        \item\label{extension-generation} $\Perf_{[c,d]}(R)$ is the smallest full subcategory of $\Perf(R)$ that is closed under extensions and contains $P[n]$ for any $P\in\cProj(R)$ and integer $n\in[c,d]$. 
        \item\label{retract-free} $\cProj(R)$ consists of exactly retracts of finite free $R$-modules. 
        \item\label{dualizable} $\Perf(R)$ consists of exactly the dualizable objects in $(\D(R),\otimes_R)$, and the dual of $\Perf_{[c,d]}(R)$ is $\Perf_{[-d,-c]}(R)$. 
    \end{enumerate}
\end{proposition}

\begin{proof}
    Note first that by definition it is easy to check that $\Perf_{\ge c}(R)$, $\Perf_{\le d}(R)$, and $\Perf_{[c,d]}(R)$ are all closed under extensions. 
    \begin{enumerate}
        \item Translating, we may assume $c=0$. Since the t-structure on $\D(R)$ is compatible with filtered colimits by \cite[Proposition 7.1.1.13]{ha}, $\Perf_{\ge0}(R)$ consists of compact objects of $\D_{\ge0}(R)$. Consider the functor $\pi_0\colon\D_{\ge0}(R)\to\D_\heartsuit(R)=\D_\heartsuit(\pi_0(R))$. It is left adjoint to the natural inclusion functor which preserves filtered colimits, so it preserves compact objects. This means that for $M\in\Perf_{\ge0}(R)$, $\pi_0(M)$ is a finitely presented classical $\pi_0(R)$-module. Choose generators $m_1,\ldots,m_r$ for $\pi_0(M)$ and lift them to $M$. They define a map $P=R^{\oplus r}\to M$, and by construction its cofiber $C$ has trivial $\pi_0$ and hence lies in $\Perf_{\ge1}(R)$. If $d>0$, then $P[1]\in\Perf_{\le d}(R)$, so by the fiber sequence $M\to C\to P[1]$ we can see that $C\in\Perf_{\le d}(R)$. 
        \item It is obvious that $P[n]\in\Perf_{[c,d]}(R)$ for $P\in\cProj(R)$ and $n\in[c,d]$. It remains to see the other direction, which for $d=c$ follows immediately from the definition, and otherwise follows from (\ref{one-step-free-resolution}) by induction. 
        \item Take $M\in\cProj(R)$ and take a map $P\to M$ from a finite free $R$-module $P$ with cofiber $C\in\Perf_{\ge1}(R)$ as in (\ref{one-step-free-resolution}). By definition, $\pi_0\Map_R(M,C)=0$, so the map $M\to C$ is nullhomotopic and thus $M$ is a retract of $P$. 
        \item Since the tensor unit $R\in\D(R)$ is compact, dualizable objects are compact. Since $\D(R)$ is generated under colimits by the translations of $R$ which are dualizable, compact objects are dualizable. The rest follows from (\ref{extension-generation}). \qedhere
    \end{enumerate}
\end{proof}

\begin{remark}
    It is not hard to see that $\Perf_{\le d}(R)$ consists of exactly those perfect $R$-complexes with tor-amplitude $\le d$ as defined in \cite[Definition 7.2.4.21]{ha}. We choose not to include the proof of this fact as we will not use it in the following. 
\end{remark}

The following slicing criterion for tor-amplitudes will be used later. 

\begin{theorem}[cf.\ {\cite[\texttt{0H89}]{stacks}}]\label{slicing-tor-amplitude}
    Let $R$ be an animated ring, $M$ a perfect $R$-complex, and $r$ an element of $R$. Let $c$ and $d$ be integers. Then $M\in\Perf_{[c,d]}(R)$ if and only if $M[1/r]:=M\otimes_RR[1/r]\in\Perf_{[c,d]}(R[1/r])$ and $M/rM:=M\otimes_RR/r\in\Perf_{[c,d]}(R/r)$, and similarly for $\Perf_{\ge c}$ and $\Perf_{\le d}$. 
\end{theorem}

\begin{proof}
    It suffices to prove the theorem for $\Perf_{\ge c}$ and $\Perf_{\le d}$ separately, for which we may assume $c=d=0$ without loss of generality. Also, the ``only if'' part is obvious, so we only need to prove the ``if'' part. 

    Assume that $M[1/r]\in\Perf_{\ge0}(R[1/r])$ and $M/rM\in\Perf_{\ge0}(R/r)$. If $M=0$ there is nothing to prove. Otherwise, let $a\in\Z$ be the largest integer so that $M\in\Perf_{\ge a}(R)$, or in other words $\pi_a(M)$ is the lowest nonzero homotopy of $M$. Then Proposition \ref{perfect-c-d}(\ref{one-step-free-resolution}) implies that $\pi_a(M)$ is a finitely presented classical $\pi_0(R)$-module. By the exactness of inverting $r$, we see that $\pi_{<a}(M[1/r])=0$ and $\pi_a(M[1/r])=\pi_a(M)[1/r]$. By the long exact sequence of the fiber sequence $M\xrightarrow{r}M\to M/rM$, we see that $\pi_{<a}(M/rM)=0$ and $\pi_a(M/rM)$ is the quotient of the classical module $\pi_a(M)$ by its submodule $r\pi_a(M)$. Now finite presentation and nonzeroness of $\pi_a(M)$ easily imply that either $\pi_a(M[1/r])$ or $\pi_a(M/rM)$ must be nonzero, and hence $a\ge0$ as we want. 

    For the other half, we use the duality functor $(-)^\vee:=\map_R(-,R)$ which swaps $\Perf_{\ge0}(R)$ and $\Perf_{\le0}(R)$. Assume that $M[1/r]\in\Perf_{\le0}(R[1/r])$ and $M/rM\in\Perf_{\le0}(R/r)$. Then for any $R$-algebra $A$,
    $$M^\vee\otimes_RA=\map_R(M,A)=\map_A(M\otimes_RA,A)=(M\otimes_RA)^\vee$$
    by dualizability of perfect complexes. In particular, taking $A$ as $R[1/r]$ and $R/r$ implies that $M^\vee[1/r]\in\Perf_{\ge0}(R[1/r])$ and $M^\vee/rM^\vee\in\Perf_{\ge0}(R/r)$. By the previous paragraph, we see that $M^\vee\in\Perf_{\ge0}(R)$ and hence $M\in\Perf_{\le0}(R)$. 
\end{proof}

\begin{definition}[Stably free complexes]
    For a ring $R$, we say that a perfect $R$-complex $M$ is \emph{stably free} if its class $[M]\in\K_0(R):=\K_0(\Perf(R))$ is an integer multiple of $[R]$. Here $\K_0$ of a small stable category means the classical abelian group generated by objects of the category subject to the relations $[Y]=[X]+[Z]$ for all fiber sequences $X\to Y\to Z$. 
\end{definition}

The following proposition is actually a simple special case of \cite[Theorem 6.0.1]{elmanto-sosnilo-weight}, but we give an elementary proof here for the convenience of the reader. 

\begin{proposition}\label{projective-K0-comparison}
    Let $R$ be a ring. Then $\K_0(R)$ is the group completion of the monoid of isomorphism classes of $\cProj(R)$ under the direct sum. 
\end{proposition}

\begin{proof}
    Denote the group completion by $\K_0^{\oplus}(R)$. It has an obvious map to $\K_0(R)$ which is surjective by Proposition \ref{perfect-c-d}(\ref{extension-generation}). It suffices to define a left inverse $\K_0(R)\to\K_0^{\oplus}(R)$ to this map. In other words, we need to assign an element $[X]^\oplus$ for $X\in\Perf(R)$ such that $[Y]^\oplus=[X]^\oplus+[Z]^\oplus$ for fiber sequences $X\to Y\to Z$ in $\Perf(R)$. To this end, we will for $c,d\in\Z$ do such assignments for $\Perf_{[c,d]}(R)$ that satisfy the above equality for fiber sequences in $\Perf_{[c,d]}(R)$, and show that our assignments are independent of $c,d\in\Z$. 
    
    First fix $d\in\Z$, and do downward induction on $c$ to define $[X]^\oplus$ for $X\in\Perf_{[c,d]}(R)$ and to show $[Y]^\oplus=[X]^\oplus+[Z]^\oplus$ for fiber sequences $X\to Y\to Z$ in $\Perf_{[c,d]}(R)$. If $c=d$, then $X=P[d]$ for $P\in\cProj(R)$, and we just define $[X]^\oplus=(-1)^d[P]$. Its compatibility with fiber sequences follows from the fact that fiber sequences in $\cProj(R)$ split. Otherwise, take a map $f\colon P[c]\to X$ with $P\in\cProj(R)$ and cofiber in $\Perf_{[c+1,d]}(R)$, and define $[X]^\oplus=(-1)^c[P]+[\cofib(f)]^\oplus$. We first show that this is independent of $f$: take another such map $g\colon Q[c]\to X$ and let $(f,g)\colon P[c]\oplus Q[c]\to X$ be the direct sum. Then the commutative diagram
    $$\begin{tikzcd}
        &P[c]\ar[r,equal]\ar[d]&P[c]\ar[d]\\
        Q[c]\ar[r]\ar[d,equal]&X\ar[r]\ar[d]&\cofib(g)\ar[d]\\
        Q[c]\ar[r]&\cofib(f)\ar[r]&\cofib(f,g)
    \end{tikzcd}$$
    shows $(-1)^c[P]+[\cofib(f)]^\oplus=(-1)^c[P\oplus Q]+[\cofib(f,g)]^\oplus=(-1)^c[Q]+[\cofib(g)]^\oplus$ by induction hypothesis. In particular, it agrees with the existing definition on $\Perf_{[c+1,d]}(R)$. For a fiber sequence $X\to Y\to Z$ in $\Perf_{[c,d]}(R)$, take $f\colon P[c]\to X$ and $g\colon Q[c]\to Z$ with $P,Q\in\cProj(R)$ and cofibers in $\Perf_{[c+1,d]}(R)$. Since $\cofib(Y\to Z)=X[1]\in\D_{>c}(R)$, one can lift $g$ to a map $Q[c]\to Y$ and form the commutative diagram
    $$\begin{tikzcd}
        P[c]\ar[r]\ar[d]&P[c]\oplus Q[c]\ar[r]\ar[d]&Q[c]\ar[d]\\
        X\ar[r]&Y\ar[r]&Z
    \end{tikzcd}$$
    Taking vertical cofibers, the desired equality for $X\to Y\to Z$ follows by induction. 

    It remains to show that the above definition is independent of $d$. This boils down to showing that for $X=P[d]\in\Perf_{[d,d]}(R)$ and $f\colon F[d]\to X$ with $F\in\cProj(R)$ and cofiber $\cofib(f)=Q[d+1]\in\Perf_{[d+1,d+1]}$, we have $(-1)^d[P]=(-1)^d[F]+(-1)^{d+1}[Q]$. This is obvious, as in fact $F=P\oplus Q$. 
\end{proof}

The following corollary is obvious. 

\begin{corollary}\label{stably-free-corollary}
    Let $R$ be a ring and $c<d$ be integers. 
    \begin{enumerate}
        \item If $P\in\cProj(R)$ is stably free, then there exists a finite free module $Q$ such that $P\oplus Q$ is also finite free. 
        \item Every stably free complex $M\in\Perf_{[c,d]}(R)$ is a finite extension of modules of the form $R[n]$ for integers $n\in[c,d]$. 
    \end{enumerate}
\end{corollary}

\begin{definition}[Connectivity]
    For $c\in\N$, a ring map is called \emph{$c$-connective} if it is $c$-connective as an anima map. That is, a ring map $A\to B$ is $c$-connective if: 
    \begin{itemize}
        \item For $i\le c$, the map $\pi_i(A)\to\pi_i(B)$ is a surjection. 
        \item For $i<c$, the map $\pi_i(A)\to\pi_i(B)$ is an isomorphism. 
    \end{itemize}
    Note that a $0$-connective map is the same as a surjection. 
\end{definition}

\begin{proposition}\label{lift-module}
    Let $c\in\N$ and $A\to B$ be a $c$-connective ring map. Then: 
    \begin{enumerate}
        \item\label{lift-mapping-anima} For all $n\in\N$, the functor $\Perf_{[0,n]}(A)\to\Perf_{[0,n]}(B)$ is $(c-n)$-connective on mapping animas. 
        \item The functor $\Perf_{[0,c+1]}(A)\to\Perf_{[0,c+1]}(B)$ hits every object whose $\K_0$ class is hit. In particular, it hits every stably free complex. 
    \end{enumerate}
\end{proposition}

\begin{proof}
    \begin{enumerate}
        \item Let $X,Y\in\Perf_{[0,n]}(A)$ and let $M=\map_A(X,Y)$. By Proposition \ref{perfect-c-d}(\ref{dualizable}), $M=X^\vee\otimes_AY\in\Perf_{[-n,n]}(A)$, and we want to show that $\tau_{\ge0}(M)\to\tau_{\ge0}(M\otimes_AB)$ is $(c-n)$-connective. For this we only need to show that $M\to M\otimes_AB$ is a $(c-n)$-connective map of spectra, which is clear as $M$ is $(-n)$-connective and by assumption $\fib(A\to B)$ is $c$-connective. 
        \item Do induction on $c$. For $c>0$, the claim easily follows from the $n=c$ case of (\ref{lift-mapping-anima}) and Proposition \ref{perfect-c-d}(\ref{one-step-free-resolution}). For $c=0$, we need to lift any module $\bar{X}\in\Perf_{[0,1]}(B)$ to $\Perf_{[0,1]}(A)$. By Proposition \ref{perfect-c-d}(\ref{one-step-free-resolution}), we can write $\bar{X}=\cofib(\bar{P}\to\bar{Q})$ where $\bar{P},\bar{Q}\in\cProj(B)$ and $\bar{Q}$ is finite free. Now the $\K_0$ class of $\bar{P}$ is hit, so by Proposition \ref{projective-K0-comparison} and Proposition \ref{perfect-c-d}(\ref{retract-free}), we can add a finite free module to both $\bar{P}$ and $\bar{Q}$, and assume that they come from $P,Q\in\cProj(A)$, and then the claim follows from (\ref{lift-mapping-anima}). \qedhere
    \end{enumerate}
\end{proof}

The following construction is crucial. 

\begin{definition}[Quotient algebra]\label{quotient-algebra}
    Let $R$ be a ring and $M$ be an $R$-module with a map $f\colon M\to R$. Form the pushout square
    $$\begin{tikzcd}
        \LSym_R(M)\ar[r,"\LSym(0)"]\ar[d,"f"']&R\ar[d]\\
        R\ar[r,"q"']&R\sslash M
    \end{tikzcd}$$
    of $R$-algebras. We call $q\colon R\to R\sslash M$ the \emph{quotient algebra} of $R$ by $M$. Sometimes we will use the notation $R\sslash f$ or $R\sslash(M,f)$ instead of $R\sslash M$ to avoid confusion. 
\end{definition}

\begin{remark}\label{quotient-functoriality}
    The formation of $R\sslash M$ defines a functor from the slice category of $R$-modules over $R$ to the category of $R$-algebras that preserves all colimits. Moreover, it satisfies base change, meaning that for a ring map $R\to R'$, if we set $M'=M\otimes_RR'$ and $f'=f\otimes_RR'$, then $R'\sslash M'=(R\sslash M)\otimes_RR'$. 
\end{remark}

\begin{remark}\label{quotient-yoneda}
    Since $\LSym_R$ is the free functor left adjoint to the forgetful functor from $R$-algebras to $R$-modules, it is easy to see that a map from $R\sslash M$ to an $R$-algebra $A$ is just a nullhomotopy of the $R$-module map $M\to R\to A$. More formally, $R\sslash M$ represents the functor $A\mapsto*\times_{\Map_R(M,A)}*$ on $\Alg_R$, where the two points are the zero map and the map $M\to R\to A$. 
\end{remark}

We recall some finiteness conditions on algebras. 

\begin{definition}[Compact algebras]\label{compact-algebra}
    We say that a ring map $R\to A$ is:
    \begin{itemize}
        \item \emph{compact}, if $A$ is a compact object in the category of $R$-algebras. 
        \item \emph{classically compact}, if $\pi_0(R)\to\pi_0(A)$ is finitely presented as a map of classical rings. 
    \end{itemize}
\end{definition}

\begin{remark}
    ``Compact'' is the analog of ``locally of finite presentation'' in \cite[Definition 7.2.4.26]{ha} in the animated setting. We choose this term simply because it has far fewer syllables. 
\end{remark}

\subsection{Cotangent complexes}
The cotangent complex is the higher analog of the K\"ahler differential module. Here we will not recall the definitions of square-zero extensions, cotangent complexes, and the basic deformation theory of animated rings; we refer the reader to \cite[\S 5.1.8, \S 5.1.9]{purity-flat} and \cite[\S 25.3]{sag}. Instead, we will control algebras with their cotangent complexes, just as Lurie has done in \cite[\S 7.4.3]{ha} for $\E_\infty$-rings. We will see that the cotangent complex has much more control on higher homotopy than on $\pi_0$; this can be seen as evidence of the intuition that higher homotopy should be thought of as thickening, and is something nearly linear. 

Below, for a ring $R$, when we write fibers and cofibers for maps of $R$-algebras, we mean fibers and cofibers as $R$-complexes. 

\begin{proposition}\label{cotangent-of-quotient}
    Let notations be as in Definition \ref{quotient-algebra}, and let $I=\fib(R\to R\sslash M)$ with a natural $R$-module map $M\to I$ factorizing $f\colon M\to R$. Then
    $$\L_{(R\sslash M)/R}=M[1]\otimes_R(R\sslash M),$$
    and the composition $M[1]\to I[1]\to\L_{(R\sslash M)/R}=M[1]\otimes_R(R\sslash M)$ of the natural map and the universal derivation can be naturally identified with the base change map of $M[1]$ along $R\to R\sslash M$; in other words, the map $M\to I$ has a retract after base change to $R\sslash M$ given by the scalar extension of the universal derivation. 
\end{proposition}

\begin{proof}
    Recall that $\L_{\LSym_R(M)/R}=M\otimes_R\LSym_R(M)$. Consider the composition $R\to\LSym_R(M)\to R$, where the second arrow is $\LSym_R(M\to0)$. Its cotangent fiber sequence reads
    $$M\to0\to\L_{R/\LSym_R(M)},$$
    which implies $\L_{R/\LSym_R(M)}=M[1]$, so the equality follows by base change. To give the identification, since both $M[1]$ and $\L_{(R\sslash M)/R}=M[1]\otimes_R(R\sslash M)$ preserve colimits with respect to $(M,f)$, we are reduced to the case $M=R$, and then $f$ can be seen as an element in $R$; now by base change we are reduced to the universal such case, namely $R=\Z[x]$ and $f=x$, and we can compute
    $$\L_{\Z/\Z[x]}=\cofib(\L_{\Z[x]/\Z}\otimes_{\Z[x]}\Z\to\L_{\Z/\Z})=\cofib(\Z\d x\to0)=\Z\d x[1],$$
    where the universal derivation $\d\colon\cofib(\Z[x]\to\Z)\to\L_{\Z/\Z[x]}$ takes $x[1]$ to $\d x[1]$ and $x^n[1]$ to $0$ for $n>1$, as expected. 
\end{proof}

\begin{corollary}\label{cotangent-over-quotient}
    Let $R\to A$ be a ring map with cofiber $C$. Let $M$ be an $R$-module with a map $g\colon M\to C[-1]$. Then $R\to A$ naturally factors through $R\sslash M$, and
    $$\L_{A/(R\sslash M)}=\cofib(M[1]\otimes_RA\to\L_{A/R}),$$
    where the map $M[1]\otimes_RA\to\L_{A/R}$ is the scalar extension along $R\to A$ of the composition $M[1]\to C\to\L_{A/R}$ of $g[1]$ and the universal derivation. 
\end{corollary}

\begin{proof}
    That $R\to A$ naturally factors through $R\sslash M$ is Remark \ref{quotient-yoneda}. The rest follows from the proposition and the cotangent fiber sequence of $R\to R\sslash M\to A$, since the universal derivation is functorial with respect to maps of $R$-algebras. 
\end{proof}

The following notion is central in the paper. 

\begin{definition}[$d$-smoothness]\label{def-smooth}
    Let $d$ be a natural number. A ring map $R\to A$ is \emph{$d$-smooth} if it is compact and $\L_{A/R}\in\Perf_{[0,d]}(A)$. A ring map is \emph{ind-$d$-smooth} if it is a filtered colimit of $d$-smooth ring maps. 
\end{definition}

\begin{remark}\label{property-smooth}
    Since the cotangent complex of a compact map is perfect, every compact map is $d$-smooth for some $d\in\N$. For any $d\in\N$ and any ring $R$, it is easy to see that:
    \begin{itemize}
        \item $d$-smooth maps are closed under base change and finite composition.
        \item Ind-$d$-smooth maps are closed under base change, composition, and filtered colimits. 
        \item Maps of $d$-smooth $R$-algebras are $(d+1)$-smooth. 
        \item Maps of ind-$d$-smooth $R$-algebras are ind-$(d+1)$-smooth. 
        \item For $M\in\Perf_{[0,d]}(R)$ with a map to $R$, $R\sslash M$ is $(d+1)$-smooth over $R$. 
    \end{itemize}
\end{remark}

To control $d$-smooth maps, we will use the following proposition in \cite{sag}, whose statement is reproduced below for the convenience of the reader: 

\begin{proposition}[{\cite[Proposition 25.3.6.1]{sag}}]\label{crucial-connectivity}
    Let $\varphi\colon R\to A$ be a ring map. The scalar extension $\cofib(\varphi)\otimes_RA\to\L_{A/R}$ of the universal derivation is: 
    \begin{itemize}
        \item surjective;
        \item $2$-connective, if $\varphi$ is surjective;
        \item $(c+3)$-connective, if $\varphi$ is $c$-connective for some $c\in\Z_+$. 
    \end{itemize}
\end{proposition}

\begin{corollary}\label{cotangent-connectivity}
    Let $c$ be a natural number and $\varphi\colon R\to A$ be a $c$-connective ring map. Then $\L_{A/R}$ is $(c+1)$-connective. The converse holds provided that $\varphi$ is a classically compact surjection and that there is no idempotent $e\in R$ except $1$ with $A$ an $R[e^{-1}]$-algebra; in particular, in this case if $\L_{A/R}=0$ then $A=R$. 
\end{corollary}

\begin{proof}
    If $\varphi$ is $c$-connective, the proposition implies that the cofiber of $\cofib(\varphi)\to\L_{A/R}$ is $(c+1)$-connective, so $\L_{A/R}$ is also $(c+1)$-connective. Conversely if $\L_{A/R}$ is $(c+1)$-connective, let $c'$ be the largest integer so that $\varphi$ is $c'$-connective, and we want to see that $c'\ge c$. By the assumption that $\varphi$ is surjective, $c'\ge0$. If $c=0$, there is nothing to prove. If $c>0$, then since $\fib(\varphi)\otimes_RA$ is $c$-connective, we have $\pi_0(\fib(\pi_0(R)\to\pi_0(A))\otimes_{\pi_0(R)}\pi_0(A))=0$, so $\pi_0(A)$ is the quotient of $\pi_0(R)$ by a classically idempotent ideal. Now the assumptions force $\pi_0(A)=\pi_0(R)$, so the map $\cofib(\varphi)\to\cofib(\varphi)\otimes_RA$ induces an isomorphism on the lowest nonzero homotopy group, namely $\pi_{c'+1}$. Since $c'\ge0$, the proposition implies that the map $\cofib(\varphi)\otimes_RA\to\L_{A/R}$ is $(c'+2)$-connective, so the lowest nonzero homotopy group of $\L_{A/R}$ is also $\pi_{c'+1}$, and thus $c'\ge c$ because $\L_{A/R}$ is $(c+1)$-connective. 
\end{proof}

\begin{lemma}\label{cellular-replacement}
    Let $R$ be a ring and $A$ be a compact $R$-algebra. Then there is a smooth ring map $A\to B$ admitting a retract, such that $\L_{B/R}$ is stably free. We can also replace the ``smooth'' in the previous sentence by ``$d$-connective $(d+1)$-smooth'' for any natural number $d$. 
\end{lemma}

\begin{proof}
    By Proposition \ref{projective-K0-comparison}, there is some $M\in\cProj(A)$ such that $M\oplus\L_{A/R}$ is stably free. Then one can easily see that $B=\LSym_A(M)$ has the desired properties. Similarly, if one takes $M\in\cProj(A)$ with $M[d+1]\oplus\L_{A/R}$ stably free, then $\LSym_A(M[d+1])$ has the desired properties for the final claim. 
\end{proof}

\begin{lemma}\label{weak-quotient}
    Let $c$ be a natural number and $R\to A$ be a $c$-connective ring map with fiber $I$. Let $N\in\Perf_{[0,\min\{c,2\}]}(R)$ with a map $\bar{g}\colon N[c]\otimes_RA\to\L_{A/R}[-1]$. Then there is a map $g\colon N[c]\to I$ such that the natural $R$-algebra map $R\sslash N[c]\to A$ induces $\bar{g}[1]$ on cotangent complexes. 
\end{lemma}

\begin{proof}
    By Proposition \ref{cotangent-of-quotient}, it suffices to produce a dashed arrow making
    $$\begin{tikzcd}
        N[c]\ar[d,dashed,"g"']\ar[rr]&&N[c]\otimes_RA\ar[d,"\bar{g}"]\\
        I\ar[r]&I\otimes_RA\ar[r]&\L_{A/R}[-1]
    \end{tikzcd}$$
    commute. Since both $I$ and $R\to A$ are $c$-connective, the lower left map is $2c$-connective; by Proposition \ref{crucial-connectivity}, the lower right map is $(c+2)$-connective when $c>0$ and $1$-connective when $c=0$; so the cofiber of their composition is $(c+1+\min\{c,2\})$-connective, and thus the map from $N[c]$ to the cofiber has a nullhomotopy, giving the desired dashed arrow. 
\end{proof}

The following theorem is the analog of \cite[Theorem 7.4.3.18]{ha} in the setting of animated rings, saying that a classically compact ring map with a perfect cotangent complex is compact. 

\begin{theorem}\label{factorize-smooth}
    Let $c\le d$ be natural numbers. A ring map $R\to A$ is $d$-smooth if and only if it is classically compact and $\L_{A/R}\in\Perf_{[0,d]}(A)$. Assume that $d>0$, that $\fib(R\to A)$ is $c$-connective, and if $c=0$ and $d=1$ assume further that $\L_{A/R}[-1]\in\cProj(A)$ can be lifted to $\cProj(R)$. Then there is a factorization
    $$R=A_c\to A_{c+1}\to\cdots\to A_d=A,$$
    where $A_{i+1}=A_i\sslash M_i[i]$ for some $M_i\in\cProj(A_i)$ and some map $M_i[i]\to A_i$. In this case, one can arrange so that each $M_i$ except $M_{d-1}$ is free; if furthermore $c<d-1$ and $\L_{A/R}$ is stably free, then one can arrange so that each $M_i$ is free.
\end{theorem}

\begin{proof}
    To prove the first claim, by definition we only need to show that a classically compact ring map with a perfect cotangent complex is compact. Therefore, replacing $R$ by a finitely generated polynomial $R$-algebra, we can assume that $R\to A$ is surjective, which reduces the first claim to the second claim. 

    Now do downward induction on $c$. If $c=d$, by Remark \ref{property-smooth}, $\L_{A/R}$ is trivial, so by Corollary \ref{cotangent-connectivity} we have $A=R$, because $c=d>0$ implies $\pi_0(R)=\pi_0(A)$. If $c=d-1$, we can take $N\in\cProj(R)$, an isomorphism $N[c]\otimes_RA\to\L_{A/R}[-1]$, and apply Lemma \ref{weak-quotient}. Then the resulting map $R\sslash N[c]\to A$ is an isomorphism by Proposition \ref{cotangent-of-quotient} and Corollary \ref{cotangent-connectivity}, which implies the claim. If $c<d-1$, by Proposition \ref{perfect-c-d}(\ref{one-step-free-resolution}) there is a finite free $R$-module $N$ and a map $N[c+1]\otimes_RA\to\L_{A/R}$ with cofiber in $\Perf_{[c+2,d]}(A)$, and if $c=d-2$ and $\L_{A/R}$ is stably free, by Proposition \ref{projective-K0-comparison} we can make this cofiber a translation of a free module. Then the claim follows from Lemma \ref{weak-quotient} and Corollary \ref{cotangent-of-quotient} by induction. 
\end{proof}

\begin{corollary}\label{localization-extending}
    Let $R$ be a ring, $S\subseteq R$ be a multiplicative submonoid, and $A'$ be a compact $S^{-1}R$-algebra with $\L_{A'/S^{-1}R}$ stably free. Let $\Lambda$ be an $R$-algebra with an $S^{-1}R$-homomorphism $\psi\colon A'\to S^{-1}\Lambda$. Then there is a compact $R$-algebra $A$ with a map $\varphi\colon A\to\Lambda$, such that $\L_{A/R}$ stably free and $S^{-1}\varphi=\psi$. 
\end{corollary}

\begin{proof}
    Take an integer $d>1$ such that $A'$ is $d$-smooth over $S^{-1}R$. By the theorem, there is a factorization
    $$S^{-1}R\to A'_0\to A'_1\to\cdots\to A'_d=A',$$
    where $A'_0$ is a finitely generated polynomial $S^{-1}R$-algebra, and $A'_{i+1}=A'_i\sslash A'_i[i]^{\oplus r}$ for some $r\in\N$ and some map $A'_i[i]^{\oplus r}\to A'_i$. Now it suffices to prove the corollary for each individual step of the factorization. For $S^{-1}R\to A'_0$ this is trivial. For other steps, it suffices to lift the $A'_i$-module map $A'_i[i]^{\oplus r}\to A'_i$ along with the nullhomotopy of the composite $A'_i[i]^{\oplus r}\to A'_i\to S^{-1}\Lambda$ to an $A_i$-module map $A_i[i]^{\oplus r}\to A_i$ along with a nullhomotopy of the composite $A_i[i]^{\oplus r}\to A_i\to\Lambda$, up to multiplication by an element in $S$. This is easy, because 
    $$\map_{A'_i}(A'_i[i]^{\oplus r},A'_i)=A'_i[-i]^{\oplus r}=S^{-1}A_i[-i]^{\oplus r}=S^{-1}\map_{A_i}(A_i[i]^{\oplus r},A_i),$$
    and similarly $\map_{A'_i}(A'_i[i]^{\oplus r},S^{-1}\Lambda)=S^{-1}\map_{A_i}(A_i[i]^{\oplus r},\Lambda)$.
\end{proof}

The following theorem shows that we can do better for a compact surjection with cotangent complex in a small range, writing it as just one quotient. 

\begin{theorem}\label{strong-quotient}
    Let $c$ be a natural number and $R\to A$ be a $c$-connective $(c+2+\min\{c,1\})$-smooth ring map. If $c=0$, assume that the $\K_0$ class of $\L_{A/R}$ can be lifted to $\K_0(R)$. Then there is an $M\in\Perf_{[0,1+\min\{c,1\}]}(R)$ with a map $f\colon M[c]\to R$ such that $A=R\sslash M[c]$. 
\end{theorem}

\begin{proof}
    Without loss of generality, we can assume that there is no nontrivial idempotent $e\in R$ with $A$ an $R[e^{-1}]$-algebra. Set $I=\fib(R\to A)$. By Definition \ref{def-smooth} and Corollary \ref{cotangent-connectivity}, $\L_{A/R}[-1]\in\Perf_{[c,c+1+\min\{c,1\}]}(A)$. When $c>1$, by Proposition \ref{lift-module} we can take $M\in\Perf_{[0,2]}(R)$ such that $M[c]\otimes_RA=\L_{A/R}[-1]$. Take $g\colon M[c]\to I$ as in Lemma \ref{weak-quotient}; then the final claim of Corollary \ref{cotangent-connectivity} implies that $A=R\sslash M[c]$ and we are done. When $c\le1$, take $N\in\Perf_{[0,c]}(R)$ and $\bar{g}\colon N[c]\otimes_RA\to\L_{A/R}[-1]$ such that $\bar{K}=\cofib(\bar{g})[-2c-1]\in\cProj(A)$ can be lifted to some $K\in\cProj(R)$. Take $g\colon N[c]\to I$ as in Lemma \ref{weak-quotient}; then
    $$\L_{A/(R\sslash N[c])}=\cofib(\bar{g}[1])=\bar{K}[2c+2].$$
    Use Lemma \ref{weak-quotient} again, now with $R\sslash N[c]\to A$ and $K$, to get a map $h\colon K[2c+1]\to R\sslash N[c]$. The final claim of Corollary \ref{cotangent-connectivity} implies that $A=(R\sslash N[c])\sslash K[2c+1]$; to present $A$ as one quotient instead of two, we have to splice $K$ and $N$ together. So set $J=\fib(R\to R\sslash N[c])$ and look at the cofiber $C$ of the shift of the natural map $N[c]\to J$ as in the diagram
    $$\begin{tikzcd}
        K[2c+1]\ar[d,dashed]\ar[r,"h"]&R\sslash N[c]\ar[d]\\
        N[c+1]\ar[r]&J[1]\ar[r]&C
    \end{tikzcd}$$
    When $c=1$, by the final claim of Proposition \ref{cotangent-of-quotient} and Proposition \ref{crucial-connectivity}, $C\otimes_R(R\sslash N[1])$ is $4$-connective, so $C$ must also be $4$-connective. When $c=0$, since $N\to J$ is surjective, $C=\cofib(N[1]\to J[1])$ is $2$-connective. In either case the map from $K[2c+1]$ to $C$ has a nullhomotopy, and we can fill in the dash arrow. Now let $M=\cofib(K[c]\to N)$ be the cofiber of a shift of the dashed arrow. Taking vertical fibers gives a map $f\colon M[c]\to R$, and it is easy to see that $R\sslash M[c]=(R\sslash N[c])\sslash K[2c+1]$ by taking quotient algebras horizontally in the following diagram and computing cotangent complexes. 
    $$\begin{tikzcd}
        M[c]\ar[r,"f"]\ar[d]&R\ar[d]\\
        K[2c+1]\ar[r,"h"']&R\sslash N[c]
    \end{tikzcd}$$
\end{proof}

\subsection{Blowup algebras}
We also need the basic theory of the animated analog of blowup algebras, also known as Rees algebras (cf.\ \cite[\texttt{052P}]{stacks}). Throughout the subsection, let $R\to A$ be a surjective ring map, and let $R[t]$ be the graded polynomial algebra with $\deg(t)=1$. The following is \cite[Proposition 13.3]{complex}, whose proof we reproduce here for the convenience of the reader. 

\begin{proposition}[cf.\ {\cite[Corollary 3.54]{mao-crystalline}}]\label{rees}
    There is an initial $\Z$-graded $R[t]$-algebra $\R_{A/R}$, whose reduction modulo $t$ is a graded $A$-algebra. Moreover, this reduction is $\LSym_A(\L_{A/R}[-1])$ with $\L_{A/R}[-1]$ of grade $-1$. The formation of $\R_{A/R}$ is functorial with respect to $R\to A$ and commutes with base change in $R$. If $r\in R$ and $A=R/r$, then $\R_{A/R}=R[t,u]/(tu-r)$ with $u$ of grade $-1$, and the image of $u$ under reduction modulo $t$ is the generator $\d r\in\L_{(R/r)/R}[-1]$. 
\end{proposition}

\begin{proof}
    Let $S$ be any graded $R[t]$-algebra. If $A=R/r$, then a graded $A$-algebra structure on $S/tS$ is a nullhomotopy of $r$ in $S_0/tS_{-1}$, i.e.\ a lifting of $r$ to $tS_{-1}$, which is nothing but a map $R[t,u]/(tu-r)\to S$ as graded $R[t]$-algebras. Hence, $\R_{(R/r)/R}$ exists and equals $R[t,u]/(tu-r)$. The existence of $\R_{A/R}$ in general follows from the case of $A=R/r$, because the category of $R$-algebras with surjective structure maps is presentable and generated by objects of the form $R/r$ under colimits, and thus we can construct general $\R_{A/R}$ by colimits from this special case. Both the functoriality and the compatibility with base change of $\R_{A/R}$ follow easily from its defining universal property. It remains to identify $\R_{A/R}/t\R_{A/R}$. 

    Taking the grade-$0$ part of the pushout square of graded $R$-algebras
    $$\begin{tikzcd}
        S\ar[r]\ar[d]&S/tS\ar[d]\\
        S/tS\ar[r]&S/tS\otimes_SS/tS
    \end{tikzcd}$$
    gives rise to two $R$-algebra maps 
    $$S_0/tS_{-1}\rightrightarrows(S/tS\otimes_SS/tS)_0=S_0/tS_{-1}\oplus (S_{-1}/tS_{-2})[1],$$
    i.e.\ an $R$-derivation from $S_0/tS_{-1}$ to $(S_{-1}/tS_{-2})[1]$, and hence a map
    $$\L_{(S_0/tS_{-1})/R}\to(S_{-1}/tS_{-2})[1]$$
    of $S_0/tS_{-1}$-modules. If $S/tS$ is a graded $A$-algebra, we get a map
    $$\L_{A/R}\to(S_{-1}/tS_{-2})[1]$$
    of $A$-modules, and hence a map
    $$\LSym_A(\L_{A/R}[-1])\to S/tS$$
    of graded $A$-algebras. To show that it is an isomorphism for $S=\R_{A/R}$, it suffices to check the case of $A=R/r$ since both sides commute with colimits with respect to $A$. As $\L_{(R/r)/R}[-1]=R/r$ with generator $\d r$ and $\R_{(R/r)/R}/t\R_{(R/r)/R}=(R/r)[u]$, it remains to verify that $\d r\mapsto u$. For this, it suffices to note that the map from $R\to R/r$ to $\R_{(R/r)/R}\to\R_{(R/r)/R}/t\R_{(R/r)/R}$ identifies the grade-$0$ part of the above pushout square with the pushout square
    $$\begin{tikzcd}
        R\ar[r]\ar[d]&R/r\ar[d]\\
        R/r\ar[r]&R/r\otimes_RR/r\cong R/r\oplus(R/r)u[1]
    \end{tikzcd}$$
    that defines the derivation $\d r[1]$, thus showing $\d r[1]\mapsto u[1]$ as desired. 
\end{proof}

\begin{definition}[Blowup algebras]\label{blowup-algebra}\ 
    \begin{itemize}
        \item The \emph{blowup algebra}, or the \emph{Rees algebra}, associated to the ring map $R\to A$, is the graded $R[t]$-algebra $\R_{A/R}$ appearing in Proposition \ref{rees}. 
        \item Let $r\in\fib(R\to A)$. Then functoriality gives a map $R[t,u]/(tu-r)\to\R_{A/R}$. The \emph{affine blowup algebra} of $R\to A$ with respect to $r$ is the \th{$0$} grade part of the base change of $\R_{A/R}$ along $R[t,u]/(tu-r)\to R[u^{\pm}]$. 
    \end{itemize}
\end{definition}

Note that the Rees algebra here is usually called the extended Rees algebra and is different from the one in \cite[\texttt{052Q}]{stacks}. From now on till the end of the subsection, let $r\in\fib(R\to A)$ and $B$ be the affine blowup algebra of $R\to A$ with respect to $r$. 

\begin{remark}\label{outside-exceptional}
    After inverting $r$, the maps $R/r\to A\to0$ become isomorphisms. By functoriality and compatibility with base change (along $R\to R[1/r]$), so do the maps $R[t,u]/(tu-r)\to\R_{A/R}\to R[t^{\pm}]$, and hence so does the map $R\to B$. 
\end{remark}

\begin{proposition}\label{affine-blowup-initial}
    $B$ is the initial $R$-algebra with a factorization $R/r\to A\to B/rB$ of the structure map modulo $r$. 
\end{proposition}

\begin{proof}
    Recall that the category of $R$-algebras and that of graded $R[u^{\pm}]$-algebras are equivalent. Therefore, for any $R$-algebra $C$, 
    $$\Hom_R(B,C)=\Hom_{R[u^{\pm}]}(B[u^{\pm}],C[u^{\pm}])=\Hom_{R[t,u]/(tu-r)}(\R_{A/R},C[u^{\pm}]).$$
    By Proposition \ref{rees}, the right hand side is the anima of graded $A$-algebra structures on the graded $R/r$-algebra $C[u^{\pm}]\otimes_{R[t]}R=(C/rC)[u^{\pm}]$, which is the same as the anima of $A$-algebra structures on the $R/r$-algebra $C/rC$. This proves the proposition. 
\end{proof}

\begin{proposition}\label{affine-blowup-smooth}
    Let $d$ be a natural number. If $R/r\to A$ is $(d+1)$-smooth, then so is $R\to B$, and moreover $A\to B/rB$ is $d$-smooth. 
\end{proposition}

\begin{proof}
    By definition $R/r\to A$ is compact, so by Proposition \ref{affine-blowup-initial} it is easy to see that $R\to B$ is also compact. Therefore, to show its $(d+1)$-smoothness, it suffices to check $\L_{B/R}\in\Perf_{[0,d+1]}(B)$, which can be done after inverting $r$ and after reducing modulo $r$, by Theorem \ref{slicing-tor-amplitude}. Now $R[1/r]\to B[1/r]$ is an isomorphism by Remark \ref{outside-exceptional}, and $R/r\to B/rB$ factors through $A$, so it remains to show $d$-smoothness of $A\to B/rB$. 

    By definition, $B[u^{\pm}]$ is obtained by inverting $u\in \R_{A/R}$, where $ut=r$. Thus
    $$(B/rB)[u^{\pm}]=\R_{A/R}[u^{-1}]/r=\R_{A/R}[u^{-1}]/t=\LSym_A(\L_{A/R}[-1])[u^{-1}],$$
    where $u$ goes to the image of $\d r\in\L_{(R/r)/R}[-1]$ in $\L_{A/R}[-1]$. Therefore, to show that $A\to B/rB$ is $d$-smooth, it suffices to show the same for $A[u]\to\LSym_A(\L_{A/R}[-1])$. It is clearly compact, and by the cotangent fiber sequence of $A\to A[u]\to\LSym_A(\L_{A/R}[-1])$, its cotangent complex is easily seen to be
    \begin{align*}
        &\cofib(\L_{(R/r)/R}[-1]\otimes_{R/r}A\to\L_{A/R}[-1])\otimes_A\LSym_A(\L_{A/R}[-1])\\
        &=\L_{A/(R/r)}[-1]\otimes_A\LSym_A(\L_{A/R}[-1]),
    \end{align*}
    which lies in $\Perf_{[0,d]}(\LSym_A(\L_{A/R}[-1]))$ as $R/r\to A$ is $(d+1)$-smooth. 
\end{proof}

\section{Proofs of the main results}\label{proof-main}

In this section we will prove our main theorems, Theorem \ref{main-thm-classical} and Theorem \ref{main-thm}, but before that let us address Corollary \ref{corollary-cp}. 

\begin{proof}[Proof of Corollary \ref{corollary-cp}]
    Firstly, by \cite[\texttt{07BV}]{stacks}, the map from $\O$ to its completion is ind-smooth on fraction fields, so by Theorem \ref{main-thm} it is ind-smooth. Therefore, by base change, we can reduce to the case where $\O$ is complete. 
    
    Then note that $R$ has bounded $\pi^\infty$-torsion: by \cite[\texttt{053G}]{stacks}, the quotient of $R$ by its $\pi^\infty$-torsion part is of finite presentation, so the $\pi^\infty$-torsion part is finitely generated, and hence is killed by some $\pi^n$ for $n\in\N$ large enough. 
    
    Let $\Lambda$ denote the classical $\pi$-completion of $R$. Then by the above, $\Lambda$ is also the derived $\pi$-completion of $R$, so $R/\pi^n=\Lambda/\pi^n\Lambda$ holds derivedly for all $n\in\N$. By Theorem \ref{main-thm}, it suffices to prove that $R[1/\pi]\to\Lambda[1/\pi]$ is ind-smooth. Now both rings are Noetherian and even excellent as explained in \cite[\S 1.1]{conrad-rigid}, so by Theorem \ref{popescu} it suffices to check regularity of this map, and by \cite[\texttt{07NT}]{stacks} it suffices to show that every maximal ideal of $\Lambda[1/\pi]$ gets pulled back to a maximal ideal in $R[1/\pi]$, and their completions coincide. This follows from \cite[\S 7.1.1]{non-archimedean-analysis}. 
\end{proof}

From now on, let $d$ be a natural number and $R\to\bar{R}$ be a surjective ring map with fiber $I$. Whenever we have an $R$-algebra $A$ and an $A$-module $M$, we use $\bar{A}$ and $\bar{M}$ to denote their base change to $\bar{R}$. In contrast, when we use $\bar{B}$ to denote an $\bar{R}$-algebra, we do not mean that it can be lifted to an $R$-algebra $B$. 

\subsection{The lifting problem}\label{lifting-problem}
This subsection is logically independent of the proof of Theorem \ref{main-thm} and aims at deducing Theorem \ref{main-thm-classical} from Theorem \ref{main-thm}. To this end, one needs to solve the lifting problem analogous to \cite[\texttt{07CJ}]{stacks}. Throughout the subsection, assume that $R\to\bar{R}$ is a square-zero extension whose fiber $I=\fib(R\to\bar{R})$ is $d$-connective. By this, we actually mean that we are given an $\bar{R}$-module $I$, a derivation $\bar{R}\to I[1]$, i.e.\ two maps $\bar{R}\rightrightarrows\bar{R}\oplus I[1]$ to the trivial square-zero extension, and define $R=\bar{R}\times_{\bar{R}\oplus I[1]}\bar{R}$ as the pullback of rings. 

For any $\bar{R}$-algebra $\bar{B}$, denote by $I_B$ the base change $I\otimes_{\bar{R}}\bar{B}$, no matter whether $\bar{B}$ can be lifted to an $R$-algebra $B$. 

\begin{proposition}[cf.\ {\cite[\texttt{07CM}]{stacks}}]
    Let $\Lambda$ be an $R$-algebra with $\bar{\Lambda}=\Lambda\otimes_R\bar{R}$ ind-$d$-smooth over $\bar{R}$. Then $\Lambda$ is ind-$d$-smooth over $R$. 
\end{proposition}

\begin{proof}
    It suffices to prove that, for every factorization $R\to A\to\Lambda$ with $A$ a compact $R$-algebra, one can find a further factorization $R\to A\to B\to\Lambda$ with a $d$-smooth $R$-algebra $B$. By assumption, over $\bar{R}$, there is a factorization $\bar{R}\to\bar{A}\to\bar{B}\to\bar{\Lambda}$ with $\bar{B}$ a $d$-smooth $\bar{R}$-algebra. We want to find a lift of it to $R$ after possibly enlarging $\bar{B}$. By deformation theory, this is a linear algebra problem. Let $\L_{\bar{R}}\to I[1]$ classify the square-zero extension $R\to\bar{R}$, and similarly $\L_{\bar{A}}\to I_A[1]$ and $\L_{\bar{\Lambda}}\to I_\Lambda[1]$. They fit into the commutative diagram
    $$\begin{tikzcd}
        \L_{\bar{R}}\ar[r]\ar[d]&\L_{\bar{A}}\ar[r]\ar[d]&\L_{\bar{B}}\ar[r]\ar[d,dashed]&\L_{\bar{\Lambda}}\ar[d]\\
        I[1]\ar[r]&I_A[1]\ar[r]&I_B[1]\ar[r]&I_\Lambda[1]
    \end{tikzcd}$$
    and we want to find a dashed arrow making the squares commute, after possibly enlarging $\bar{B}$. 
    
    To this end, we first find a $d$-smooth $\bar{R}$-algebra $\bar{C}$ factoring $\bar{B}\to\bar{\Lambda}$ and a dashed arrow making
    $$\begin{tikzcd}
        \L_{\bar{R}}\ar[r]\ar[d]&\L_{\bar{A}}\ar[r]\ar[d]&\L_{\bar{B}}\ar[r]\ar[d,dashed]&\L_{\bar{\Lambda}}\ar[d]\\
        I[1]\ar[r]&I_A[1]\ar[r]&I_C[1]\ar[r]&I_\Lambda[1]
    \end{tikzcd}$$
    commute. Such a dashed arrow amounts to a nullhomotopy of the $\bar{B}$-linear map 
    $$\L_{\bar{B}/\bar{A}}[-1]\to\L_{\bar{A}}\otimes_{\bar{A}}\bar{B}\to I_A[1]\otimes_{\bar{A}}\bar{B}=I_B[1]\to I_C[1]$$
    which, after composing with $I_C[1]\to I_\Lambda[1]$, coincides with the nullhomotopy given by the diagram. Remember that since $\bar{R}\to\bar{\Lambda}$ is ind-$d$-smooth, the category of such $\bar{C}$ is filtered with colimit $\bar{\Lambda}$, so the desired nullhomotopy follows from the fact that $\L_{\bar{B}/\bar{A}}[-1]=\fib(\L_{\bar{A}/\bar{R}}\otimes_{\bar{A}}\bar{B}\to\L_{\bar{B}/\bar{R}})$ is a perfect $\bar{B}$-complex. 

    Now it remains to find a dashed arrow making
    $$\begin{tikzcd}
        \L_{\bar{R}}\ar[r]\ar[d]&\L_{\bar{B}}\ar[r]\ar[rd]&\L_{\bar{C}}\ar[r]\ar[d,dashed]&\L_{\bar{\Lambda}}\ar[d]\\
        I[1]\ar[rr]&&I_C[1]\ar[r]&I_\Lambda[1]
    \end{tikzcd}$$
    commute, after possibly enlarging $\bar{C}$. Similarly, such a dashed arrow amounts to a nullhomotopy of the $\bar{C}$-linear map
    $$\L_{\bar{C}/\bar{B}}[-1]\to\L_{\bar{B}}\otimes_{\bar{B}}\bar{C}\to I_C[1]$$
    which, after composing with $I_C[1]\to I_\Lambda[1]$, coincides with the nullhomotopy given by the diagram. Since $\L_{\bar{C}/\bar{B}}[-1]=\fib(\L_{\bar{B}/\bar{R}}\otimes_{\bar{B}}\bar{C}\to\L_{\bar{C}/\bar{R}})\in\Perf_{[-1,d]}(\bar{C})$ while $I_C[1]$ is $(d+1)$-connective, the above map is indeed zero. Choose a nullhomotopy of it, corresponding to a dashed arrow $f\colon\L_{\bar{C}}\to I_C[1]$ making the middle triangle commute, defining a lift $C$ of $\bar{C}$. This nullhomotopy, after composing with $I_C[1]\to I_\Lambda[1]$, may not coincide with the preferred nullhomotopy; however, it does after composing with $\L_{\bar{C}/\bar{R}}[-1]\to\L_{\bar{C}/\bar{B}}[-1]$, because $\Map_{\bar{C}}(\L_{\bar{C}/\bar{R}}[-1],I_\Lambda[1])$ is now $2$-connective. 
    
    The above analysis implies that, ignoring $\L_{\bar{B}}$, we can make the diagram commute; in other words, we can make the right square commute, but the composition of the commutations of the square and the triangle may not coincide with the original commutation of the right trapezoid, though it does after composing with the left trapezoid. Now the difference between the two commutations constitutes a loop in $\Map_{\bar{B}}(\L_{\bar{B}},I_\Lambda[1])$ which becomes trivial in $\Map_{\bar{R}}(\L_{\bar{R}},I_\Lambda[1])$, so it is actually a loop in $\Map_{\bar{B}}(\L_{\bar{B}/\bar{R}},I_\Lambda[1])$, or equivalently a point in $\Map_{\bar{B}}(\L_{\bar{B}/\bar{R}},I_\Lambda)$. Note that $\L_{\bar{B}/\bar{R}}\in\Perf_{[0,d]}(\bar{B})$, so any point in
    $$\Map_{\bar{B}}(\L_{\bar{B}/\bar{R}},I_\Lambda)=\L_{\bar{B}/\bar{R}}^\vee[d]\otimes_{\bar{B}} I_\Lambda[-d]=\L_{\bar{B}/\bar{R}}^\vee[d]\otimes_{\bar{B}}\bar{\Lambda}\otimes_{\bar{R}}I[-d]$$
    is represented by a finite sum of elementary tensors up to a homotopy, as all the tensor factors are connective. Therefore, we can find a finitely generated polynomial $C$-algebra $D$, a map $D\to\Lambda$, and a lift of the above point to $\Map_{\bar{B}}(\L_{\bar{B}/\bar{R}},I_D)$. Now $D$ corresponds to an arrow $g\colon\L_{\bar{D}}\to I_D[1]$ inserting into the commutative square on the right, with the difference between the two commutations of the right trapezoid lifted to a loop $h\in\Omega\Map_{\bar{B}}(\L_{\bar{B}/\bar{R}},I_D[1])$. Hence we can just use the map $g$ as follows, with the commutation of the middle trapezoid modified by $h$, to conclude. 
    $$\begin{tikzcd}
        \L_{\bar{R}}\ar[r]\ar[d]&\L_{\bar{B}}\ar[r]\ar[rd]&\L_{\bar{C}}\ar[r]\ar[d,dashed,"f"]&\L_{\bar{D}}\ar[r]\ar[d,"g"]&\L_{\bar{\Lambda}}\ar[d]\\
        I[1]\ar[rr]&&I_C[1]\ar[r]&I_D[1]\ar[r]&I_\Lambda[1]
    \end{tikzcd}$$
\end{proof}

\begin{corollary}
    The following hold: 
    \begin{enumerate}
        \item\label{truncation-detect-d-smooth} If $R$ has only finitely many nonzero homotopy groups, then an $R$-algebra is ind-$d$-smooth if its base change to $\tau_{\le\max\{0,d-1\}}R$ is ind-$d$-smooth. 
        \item\label{base-change-to-R/r-detect-d-smooth} For $n\in\Z_+$ and $r\in R$, an $R/r^n$-algebra is ind-smooth if its base change to $R/r$ is ind-smooth. 
        \item\label{base-change-classical-detect-d-smooth} For $r\in R$, if $R$ is a classical ring, then an $R/r$-algebra is ind-smooth if its base change to $\pi_0(R/r)$ is ind-smooth. 
        \item Theorem \ref{main-thm} implies Theorem \ref{main-thm-classical}. 
    \end{enumerate}
\end{corollary}

\begin{proof}
    \begin{enumerate}
        \item When $R$ is $\max\{0,d-1\}$-truncated there is nothing to prove. Otherwise, assume that $R$ is $m$-truncated for $m>\max\{0,d-1\}$ and we do induction on $m$. It suffices to exhibit $R\to\tau_{<m}(R)$ as a square-zero extension. This is done by the pullback square
        $$\begin{tikzcd}
            R\ar[r]\ar[d]&\tau_{<m}(R)\ar[d]\\
            \tau_{<m}(R)\ar[r]&\tau_{\le m+1}(\tau_{<m}(R)\otimes_R\tau_{<m}(R))
        \end{tikzcd}$$
        cf.\ \cite[Corollary 7.4.1.28]{ha}. 
        \item By induction, it suffices to exhibit $R/r^{n+1}\to R/r^n$ as a square-zero extension. For this, we can assume that $R=\Z[x]$ and $r=x$ by base change. Now the claim follows from the pullback square
        $$\begin{tikzcd}
            \Z[x]/x^{n+1}\ar[r]\ar[d]&\Z[x]/x^n\ar[d]\\
            \Z[x]/x^n\ar[r]&\tau_{\le1}(\Z[x]/x^n\otimes_{\Z[x]/x^{n+1}}\Z[x]/x^n)
        \end{tikzcd}$$
        cf.\ \cite[Example 7.4.1.21]{ha}. 
        \item This is a special case of (\ref{truncation-detect-d-smooth}) because in this case $R/r$ is $1$-truncated. 
        \item Combine (\ref{base-change-to-R/r-detect-d-smooth}) and (\ref{base-change-classical-detect-d-smooth}). \qedhere
    \end{enumerate}
\end{proof}

\subsection{The lifting lemma}
This subsection deals with the following lemma. Here, thanks to animated rings, one no longer needs to be fiendishly clever as in \cite[\texttt{07CN}]{stacks}. Throughout the subsection, fix $r\in R$ and set $\bar{R}=R/r$. 

\begin{lemma}[cf.\ {\cite[\texttt{07CP}]{stacks}}]\label{lifting-lemma}
    Let $\Lambda$ be an $R$-algebra. Let $\bar{C}$ be a $(d+1)$-smooth $\bar{R}$-algebra with a factorization $\bar{R}\to\bar{C}\to\bar{\Lambda}$. Then there is an $R$-algebra $D$ and a commutative diagram
    $$\begin{tikzcd}
        R\ar[rr]\ar[d]&&D\ar[r]\ar[d]&\Lambda\ar[d]\\
        \bar{R}\ar[r]&\bar{C}\ar[r]&\bar{D}\ar[r]&\bar{\Lambda}
    \end{tikzcd}$$
    satisfying that $D[1/r]$ is smooth over $R[1/r]$ and that $\bar{D}$ is $d$-smooth over $\bar{C}$. In fact, if $\bar{R}\to\bar{C}$ is surjective, there is an initial $R$-algebra $D$ with a factorization $\bar{R}\to\bar{C}\to\bar{D}$; this $D$ satisfies the above conditions with $D[1/r]=R[1/r]$.  
\end{lemma}

\begin{proof}
    Replacing $R$ by a finitely generated polynomial $R$-algebra, we can assume that $\bar{R}\to\bar{C}$ is surjective. In this case everything reduces to the final claim, which follows from Remark \ref{outside-exceptional} and Propositions \ref{affine-blowup-initial} and \ref{affine-blowup-smooth}. 
\end{proof}

\subsection{The desingularization lemma}
This subsection deals with the desingularization lemma as in \cite[\texttt{07CQ}]{stacks}. Here animated rings do not help that much; one still needs the fiendishly clever construction given there. 
Throughout the subsection, fix $r\in R$ and set $\bar{R}=R/r$. 

\begin{lemma}\label{desing-initial-quotient}
    Let $M$ be an $R$-module with a map $f\colon M\to R$ and an $R$-algebra map $R\sslash f\to\bar{R}$. Then there is a natural $g\in\Map_R(M,R)$ with $f=rg$, and $R\sslash g$ is the initial $R$-algebra $\Lambda$ with a commutative diagram
    $$\begin{tikzcd}
        R\ar[r]\ar[d]&R\sslash f\ar[r]\ar[ld]&\Lambda\ar[d]\\
        \bar{R}\ar[rr]&&\bar{\Lambda}
    \end{tikzcd}$$
\end{lemma}

\begin{proof}
    By Remark \ref{quotient-yoneda}, an $R$-algebra map $R\sslash f\to\bar{R}$ is a nullhomotopy of the composition $M\to R\to\bar{R}$, which is the same as a lift of $f$ along $r\colon R\to R$, giving the $g$ in the statement. By the same reason, for a fixed $\Lambda$, such a commutative diagram is a nullhomotopy $z$ of the composition $M\to R\to\Lambda$ whose further composition to $M\to R\to\bar{\Lambda}$ is the given nullhomotopy, which is the same thing as a lift of $z$ along $r\colon\Lambda\to\Lambda$. This is by definition represented by $R\sslash g$. 
\end{proof}

\begin{lemma}\label{desing-initial-poly}
    For any $n\in\N$, the initial $R$-algebra $\Lambda$ with a commutative diagram
    $$\begin{tikzcd}
        R\ar[r]\ar[d]&R[x_1,\ldots,x_n]\ar[r]\ar[ld,"{(0,\ldots,0)}"]&\Lambda\ar[d]\\
        \bar{R}\ar[rr]&&\bar{\Lambda}
    \end{tikzcd}$$
    is given by $\Lambda=R[y_1,\ldots,y_n]$, with the upper right arrow sending $x_i$ to $ry_i$. 
\end{lemma}

\begin{proof}
    Omitted. It is similar to and simpler than Lemma \ref{desing-initial-quotient}. 
\end{proof}

\begin{lemma}[cf.\ {\cite[\texttt{07CR}]{stacks}}]\label{desing-lemma-0}
    Let $A$ and $\Lambda$ be $R$-algebras with a commutative diagram
    $$\begin{tikzcd}
        R\ar[r]\ar[d]&A\ar[r]\ar[ld]&\Lambda\ar[d]\\
        \bar{R}\ar[rr]&&\bar{\Lambda}
    \end{tikzcd}$$
    Let $p\in R$ with $p^4=r$. Suppose that $R\to A$ is $1$-smooth, that $[\L_{A/R}]=[A^{\oplus k}]\in\K_0(A)$ for some $k\in\N$, and that there are maps $s\colon A^{\oplus k}\to\L_{A/R}$ and $t\colon\L_{A/R}\to A^{\oplus k}$ with $s\circ t=p\in\End_A(\L_{A/R})$. Then there is a smooth $R$-algebra $B$ factorizing $A\to\Lambda$, with $B[1/r]$ smooth over $A[1/r]$. 
\end{lemma}

\begin{proof}
    By assumption, we can take a surjection $P=R[x_1,\ldots,x_n]\to A$ with $\L_{A/P}[-1]$ finite free of rank $m=n-k$. Translating, we can assume that every $x_i$ is mapped to $0$ under $A\to\bar{R}$. By Theorem \ref{factorize-smooth}, there is a map $f\colon P^{\oplus m}\to P$ with $A=P\sslash f$. Combining Lemmas \ref{desing-initial-quotient} and \ref{desing-initial-poly}, we see that the initial such commutative diagram is given by
    $$\begin{tikzcd}
        R\ar[r]\ar[d]&A=P\sslash f\ar[r]\ar[ld]&\Lambda=Q\sslash g\ar[d]\\
        \bar{R}\ar[rr]&&\bar{\Lambda}
    \end{tikzcd}$$
    where $Q=R[y_1,\ldots,y_n]$, the map $P\to Q$ sends $x_i$ to $ry_i$, and $g\colon Q^{\oplus m}\to Q$ satisfies $rg=f\otimes_PQ$, or informally $g(y)=f(ry)/r$. It suffices to prove the lemma for this diagram. Now write
    \begin{equation}\label{gexpanded}
        g=g(0)+y\cdot a+(\text{higher order terms}),
    \end{equation}
    where $g(0)\colon R^{\oplus m}\to R$ and $a\colon R^{\oplus m}\to R^{\oplus n}$ are certain matrices; then
    \begin{equation}\label{fexpanded}
        f=rg(0)+x\cdot a+(\text{higher order terms}).
    \end{equation}
    Note that $\L_{A/R}$ can be presented as 
    $$\cofib(P^{\oplus m}\otimes_PA\to\L_{P/R}\otimes_PA)=\cofib\left(A^{\oplus m}\to A^{\oplus n}\right),$$
    with the map given by $\partial f/\partial x$. Thus we can lift the map $s\colon A^{\oplus k}\to\L_{A/R}$ to a matrix $s\colon A^{\oplus k}\to A^{\oplus n}$, view the map $t\colon\L_{A/R}\to A^{\oplus k}$ as a matrix $t\colon A^{\oplus n}\to A^{\oplus k}$ with a nullhomotopy of $t\circ\partial f/\partial x$, and finally write the homotopy $s\circ t=p$ as a matrix $h\colon A^{\oplus n}\to A^{\oplus m}$ with homotopies $h\circ\partial f/\partial x=p$ and $\partial f/\partial x\circ h=p-s\circ t$. Lift $h$ to a matrix $h\colon P^{\oplus n}\to P^{\oplus m}$ and consider its constant term $h(0)\colon R^{\oplus n}\to R^{\oplus m}$. Recall that the constant term of $\partial f/\partial x$ is $a\colon R^{\oplus m}\to R^{\oplus n}$ as in (\ref{fexpanded}). We apply the map $A\to\bar{R}$ on the homotopy $h\circ\partial f/\partial x=p$; in view of the commutative diagram
    $$\begin{tikzcd}
        P\ar[r,"{(0,\ldots,0)}"]\ar[d]&R\ar[d]\\
        P\sslash f=A\ar[r]&\bar{R}
    \end{tikzcd}$$
    this gives a homotopy in $\End_R(R^{\oplus m})$ between $h(0)\circ a$ and $p$ after base change to $\bar{R}$, or equivalently a point $q\in\End_R(R^{\oplus m})$ with
    \begin{equation}\label{hexpanded}
        h(0)\circ a=p+rq=p(1+p^3q).
    \end{equation}
    Now consider the polynomial ring $S=R[v_1,\ldots,v_m,w_1,\ldots,w_n]$ and the ring map $P\to S$ given by $x\mapsto p^2(v\cdot h(0)+pw)$. Then by (\ref{hexpanded}), the base change of $f\colon P^{\oplus m}\to P$ along $P\to S$ is of the form 
    $$p^3(v+w\cdot a)+p^4(\text{other terms});$$
    this amounts to some $c\colon S^{\oplus m}\to S$ that becomes $v+w\cdot a$ after base change to $S/pS$ and satisfies $p^3c=f\otimes_PS$. Let $B=S\sslash c$. It clearly receives a map from $A=P\sslash f$. To define the map $B\to\Lambda$, consider the map $S\to Q$ taking all the $v_j$ to $0$ and $w_i$ to $py_i$. By construction, the base change of $c$ along $S\to Q$ is nothing but $pg$, and the identifications $pg=c\otimes_SQ$ and $p^3c=f\otimes_PS$ compose to the original identification $p^4g=f\otimes_PQ$. This gives the desired map $B=S\sslash c\to Q\sslash g=\Lambda$ as well as the desired factorization $A\to B\to\Lambda$. 

    We now verify the claimed properties of $B$. Since $p^3c=f\otimes_PS$, the map $A[1/p]\to B[1/p]$ is a base change of the map $P[1/p]\to S[1/p]$ given by $x\mapsto p^2(v\cdot h(0)+pw)$, which is clearly smooth: in fact $S[1/p]=P[1/p][v_1,\ldots,v_m]$ is polynomial, as we can write $w=p^{-1}(p^{-2}x-v\cdot h(0))$. By assumption, $\L_{A[1/p]/R[1/p]}\cong A[1/p]^{\oplus k}$, so $A[1/p]$ is smooth over $R[1/p]$, and hence so is $B[1/p]$. Thus it remains to verify that $B/pB$ is smooth over $R/p$. This is also clear: recall that $c\colon S^{\oplus m}\to S$ reduces to $v+w\cdot a$ over $S/pS$, so $B/pB=(S/pS)\sslash(v+w\cdot a)=R/p[v_1,\ldots,v_m,w_1,\ldots,w_n]\sslash(v+w\cdot a)=R/p[w_1,\ldots,w_n]$ is in fact polynomial. 
\end{proof}

\begin{lemma}[cf.\ {\cite[\texttt{07F0}]{stacks}}]\label{ultimate-desing-0}
    Let $A$ and $\Lambda$ be $R$-algebras and $\bar{C}$ be an $\bar{R}$-algebra, with a commutative diagram
    $$\begin{tikzcd}
        R\ar[r]\ar[d]&A\ar[rr]\ar[d]&&\Lambda\ar[d]\\
        \bar{R}\ar[r]&\bar{A}\ar[r]&\bar{C}\ar[r]&\bar{\Lambda}
    \end{tikzcd}$$
    Let $p\in R$ with $p^4=r$. Suppose that $\bar{R}\to\bar{C}$ is $1$-smooth, and that $R\to A$ satisfies the assumptions in Lemma \ref{desing-lemma-0}. Then there is an $R$-algebra $B$ factorizing $A\to\Lambda$ with $B[1/r]$ smooth over $A[1/r]$, such that $\bar{B}$ factorizes $\bar{C}\to\bar{\Lambda}$ and is smooth over $\bar{C}$. In particular, if $\bar{R}\to\bar{C}$ is smooth, then so is $R\to B$. 
\end{lemma}

\begin{proof}
    Apply Lemma \ref{lifting-lemma} to obtain a commutative diagram
    $$\begin{tikzcd}
        R\ar[rrr]\ar[d]&&&D\ar[r]\ar[d]&\Lambda\ar[d]\\
        \bar{R}\ar[r]&\bar{A}\ar[r]&\bar{C}\ar[r]&\bar{D}\ar[r]&\bar{\Lambda}
    \end{tikzcd}$$
    where $D[1/r]$ is smooth over $R[1/r]$ and $\bar{D}$ is smooth over $\bar{C}$. Using the map $\bar{A}\to\bar{D}$, we obtain another diagram
    $$\begin{tikzcd}
        D\ar[r]\ar[d]&A\otimes_RD\ar[r]\ar[ld]&\Lambda\ar[d]\\
        \bar{D}\ar[rr]&&\bar{\Lambda}
    \end{tikzcd}$$
    that falls into the assumptions of Lemma \ref{desing-lemma-0} by base change. Hence there is a smooth $D$-algebra $B$ factorizing $A\otimes_RD\to\Lambda$, such that $B[1/r]$ is smooth over $(A\otimes_RD)[1/r]$. It is easy to see that this $B$ fulfills our requirements. 
\end{proof}

We also need higher analogs of the above lemmas. 

\begin{lemma}\label{desing-lemma-d}
    Let $d$ be a positive integer. Let $A$ and $\Lambda$ be $R$-algebras with a commutative diagram
    $$\begin{tikzcd}
        R\ar[r]\ar[d]&A\ar[r]\ar[ld]&\Lambda\ar[d]\\
        \bar{R}\ar[rr]&&\bar{\Lambda}
    \end{tikzcd}$$
    Let $p\in R$ with $p^2=r$. Suppose that $R\to A$ is $(d-1)$-connective and $(d+1)$-smooth, and that there is some $N\in\cProj(R)$ and $A$-module maps $s\colon N_A[d]\to\L_{A/R}$ and $t\colon\L_{A/R}\to N_A[d]$ that are isomorphisms over $A[1/r]$ with $s\circ t=p\in\End_A(\L_{A/R})$ and $t\circ s=p\in\End_A(N_A[d])$; here $N_A$ denotes $N\otimes_RA$. If $d=1$, assume that the $\K_0$ class of $\L_{A/R}$ can be lifted from $\K_0(A)$ to $\K_0(R)$. Then there is a $d$-smooth $R$-algebra $B$ factorizing $A\to\Lambda$ with $B[1/r]=A[1/r]$. 
\end{lemma}

\begin{proof}
    We denote generally $X_C=X\otimes_RC$ for an $R$-module $X$ and an $R$-algebra $C$. By Theorem \ref{strong-quotient}, there is an $M\in\Perf_{[0,\min\{d,2\}]}(R)$ and an $f\colon M[d-1]\to R$ such that $A=R\sslash M[d-1]$. By Lemma \ref{desing-initial-quotient}, we see that the initial such commutative diagram in the statement is given by
    $$\begin{tikzcd}
        R\ar[r]\ar[d]&A=R\sslash f\ar[r]\ar[ld]&\Lambda=R\sslash g\ar[d]\\
        \bar{R}\ar[rr]&&\bar{\Lambda}
    \end{tikzcd}$$
    where $g\colon M[d-1]\to R$ satisfies $rg=f$. It suffices to prove the lemma for this diagram. One can lift $s[-d]\colon N_A\to\L_{A/R}[-d]=M_A$ to a map $i[1-d]\colon N\to M$, because $R\to A$ is surjective, $N\in\cProj(R)$, and $M$ is connective. Let $K=\cofib(N\to M)$ denote its cofiber. Then the assumption that $s\circ t=p$ and $t\circ s=p$ implies $r=p^2=0\in\End_A(K_A)$. 
    
    For an $R$-algebra $C$, we denote $\I_C=\fib(R\to C)$. Then the above diagram gives $R$-algebra maps $A\to\bar{R}\times_{\bar{\Lambda}}\Lambda\to\Lambda$ and hence $R$-module maps $\I_A\to\I_{\bar{R}\times_{\bar{\Lambda}}\Lambda}\to\I_\Lambda$. Note that $\I_{\bar{R}\times_{\bar{\Lambda}}\Lambda}$ is isomorphic to $\I_\Lambda$ with the natural map being multiplication by $r$; in other words the above maps are $\I_A\xrightarrow{h} \I_\Lambda\xrightarrow{r}\I_\Lambda$. 

    The homotopy $r=p^2=0\colon K_A\to K_A$ above lifts $r\colon K\to K$ to a map $K\to K\otimes_R\I_A$. Using the map $h\colon\I_A\to\I_\Lambda$ labeled above, we get maps
    $$K\to K\otimes_R\I_A\to K\otimes_R\I_\Lambda\to K$$
    where the final map is given by the canonical map $\I_\Lambda\to R$. Denote this composition by $k\in\End_R(K)$. By construction, we see that $rk=r\in\End_R(K)$ and $k_\Lambda=0\in\End_\Lambda(K_\Lambda)$. Therefore, $\cofib({\id_K}-k)\in\Perf(R)$ is supported away from $\Lambda$, and thus there is a $b\in R$ that acts as $0$ on $\cofib({\id_K}-k)$ and maps to $1\in\Lambda$. For this $b$ we have $r=0\in\End_{R[1/b]}(K[1/b])$, since $rk=r\in\End_R(K)$ and $({\id_K}-k)_{R[1/b]}\in\End_{R[1/b]}(K[1/b])$ is invertible. 
    
    Let $B=(R\sslash(g\circ i))[1/b]$ where $i\colon N[d-1]\to M[d-1]$ is the lift of $s[-1]$ chosen above. Then $\L_{B/R}=N_B[d]$, so $B$ is $d$-smooth. The $R$-module map $i$ defines a map $B\to\Lambda$ since $b\mapsto1\in\Lambda$; recall that $K=\cofib(N\to M)$, so the homotopy $r=0\in\End_{R[1/b]}(K[1/b])$ above gives a map $M[1/b]\to N[1/b]$ that factorizes $r\colon M[1/b]\to M[1/b]$, which defines a map $A\to B$, thus giving the desired factorization $A\to B\to\Lambda$. Finally, $A[1/r]=\Lambda[1/r]$ since $rg=f$, and $B[1/r]=\Lambda[1/r]$ since $i$ becomes an isomorphism on $R[1/(br)]$, so $A[1/r]=B[1/r]$. 
\end{proof}

\begin{lemma}\label{ultimate-desing-d}
    Let $d$ be a positive integer. Let $A$ and $\Lambda$ be $R$-algebras and $\bar{C}$ be an $\bar{R}$-algebra, with a commutative diagram
    $$\begin{tikzcd}
        R\ar[r]\ar[d]&A\ar[rr]\ar[d]&&\Lambda\ar[d]\\
        \bar{R}\ar[r]&\bar{A}\ar[r]&\bar{C}\ar[r]&\bar{\Lambda}
    \end{tikzcd}$$
    Suppose that $\bar{R}\to\bar{C}$ is $(d+1)$-smooth, and that $R\to A$ satisfies the assumptions in Lemma \ref{desing-lemma-d}. Then there is an $R$-algebra $B$ factorizing $A\to\Lambda$ with $B[1/r]$ smooth over $A[1/r]$, such that $\bar{B}$ factorizes $\bar{C}\to\bar{\Lambda}$ and is $d$-smooth over $\bar{C}$. In particular, if $\bar{R}\to\bar{C}$ is $d$-smooth, then so is $R\to B$. 
\end{lemma}

\begin{proof}
    Apply Lemma \ref{lifting-lemma} to obtain a commutative diagram
    $$\begin{tikzcd}
        R\ar[rrr]\ar[d]&&&D\ar[r]\ar[d]&\Lambda\ar[d]\\
        \bar{R}\ar[r]&\bar{A}\ar[r]&\bar{C}\ar[r]&\bar{D}\ar[r]&\bar{\Lambda}
    \end{tikzcd}$$
    where $D[1/r]$ is smooth over $R[1/r]$ and $\bar{D}$ is $d$-smooth over $\bar{C}$. Using the map $\bar{A}\to\bar{D}$, we obtain another diagram
    $$\begin{tikzcd}
        D\ar[r]\ar[d]&A\otimes_RD\ar[r]\ar[ld]&\Lambda\ar[d]\\
        \bar{D}\ar[rr]&&\bar{\Lambda}
    \end{tikzcd}$$
    that falls into the assumptions of Lemma \ref{desing-lemma-d} by base change. Hence there is a $d$-smooth $D$-algebra $B$ factorizing $A\otimes_RD\to\Lambda$ with $B[1/r]=(A\otimes_RD)[1/r]$. It is easy to see that this $B$ fulfills our requirements. 
\end{proof}

\subsection{Proof of the main theorem}
With all the preparations done, we now prove Theorem \ref{main-thm}. Let notations be the same as there. To show that $\Lambda$ is ind-$d$-smooth over $R$, it suffices to show that any $R$-homomorphism $A\to\Lambda$ from any compact $R$-algebra $A$ factors through some $d$-smooth $R$-algebra $B$. 

Since $S^{-1}\Lambda$ is ind-$d$-smooth over $S^{-1}R$, there is a $d$-smooth $S^{-1}R$-algebra $C'$ factorizing $S^{-1}A\to S^{-1}\Lambda$. As both $S^{-1}A$ and $C'$ are compact $S^{-1}R$-algebras, the map $S^{-1}A\to C'$ is compact. By Lemma \ref{cellular-replacement} and Corollary \ref{localization-extending}, there is a compact $A$-algebra $D$, such that $S^{-1}A\to S^{-1}D$ factors through $C'$ with $C'\to S^{-1}D$ smooth admitting a retract. Using this retract we see that $S^{-1}A\to S^{-1}\Lambda$ factors through $S^{-1}D$. Replacing $A$ by $D$, we arrive at the case where $S^{-1}R\to S^{-1}A$ is $d$-smooth. 

As $R\to A$ is compact, there is a $d'\in\N$ such that it is $d'$-smooth. If $d'\le d$ then we are already done; if $d'>d$, by induction we only need to prove that $A\to\Lambda$ factors through some $(d'-1)$-smooth $R$-algebra $B$ with $S^{-1}B$ smooth over $S^{-1}A$, as then $S^{-1}B$ will remain $d$-smooth over $S^{-1}R$. Replacing $d$ by $d'-1$, we can assume that $A$ is $(d+1)$-smooth over $R$. Using Lemma \ref{cellular-replacement} again we can assume that $\L_{A/R}$ is stably free, and that $\L_{S^{-1}A/S^{-1}R}$ is free in case $d=0$. 

If $d>0$ and we have a factorization $R\to R'\to A$ where $R\to R'$ is $(d-1)$-smooth, then we can replace $R\to A\to\Lambda$ by $R'\to A\to\Lambda$ and $S$ by its image in $R'$, and by Remark \ref{property-smooth} both the assumptions in the theorem and those created along the reductions above remain true. Therefore, we can first take $R'$ to be a polynomial $R$-algebra, and then apply Theorem \ref{factorize-smooth}, to arrive at a factorization such that $R'\to A$ is $(d-1)$-connective. In other words, when $d>0$, we can further assume that $R\to A$ is $(d-1)$-connective. Note that along the way we can let $\L_{A/R}$ remain stably free and make $\L_{S^{-1}A/S^{-1}R}[-d]$ free. 

At this point, we can take a finite free $A$-module $N$ with $S^{-1}L[d]=S^{-1}\L_{A/R}$, and eliminate denominators to get an element $p\in S$ and maps $s\colon N[d]\to\L_{A/R}$ and $t\colon\L_{A/R}\to N[d]$ with $s\circ t=p\in\End_A(\L_{A/R})$ and $t\circ s=p\in\End_A(N[d])$. Let $r=p^4$ and set $\bar{R}=R/r$. By assumption $\bar{R}\to\bar{\Lambda}$ is ind-$d$-smooth, so there is a $d$-smooth $\bar{R}$-algebra $\bar{C}$ with a factorization $\bar{A}\to\bar{C}\to\bar{\Lambda}$. Now we are in the situation of either Lemma \ref{ultimate-desing-0} or Lemma \ref{ultimate-desing-d} where the ``in particular'' assumption holds, so we win. 
\qed

\printbibliography

@online{stacks,
    shorthand    = {Stacks},
    author       = {The {Stacks Project Authors}},
    title        = {Stacks Project},
    url          = {https://stacks.math.columbia.edu},
}

@book{htt,
    shorthand = "HTT",
    url = {https://doi.org/10.1515/9781400830558},
    title = {Higher Topos Theory (AM-170)},
    author = {Jacob Lurie},
    publisher = {Princeton University Press},
    address = {Princeton},
    doi = {doi:10.1515/9781400830558},
    isbn = {9781400830558},
    year = {2009},
}

@online{ha,
    shorthand = "HA",
    author = "Jacob Lurie",
    title = "Higher Algebra",
    url = "https://www.math.ias.edu/~lurie/papers/HA.pdf",
    year = "2017"
}

@online{sag,
    shorthand = "SAG",
    author = "Jacob Lurie",
    title = "Spectral Algebraic Geometry",
    url = "https://www.math.ias.edu/~lurie/papers/SAG-rootfile.pdf",
    year = "2018"
}

@article{purity-flat,
    author = {Kęstutis Česnavičius and Peter Scholze},
    title = {{Purity for flat cohomology}},
    volume = {199},
    journal = {Annals of Mathematics},
    number = {1},
    publisher = {Department of Mathematics of Princeton University},
    pages = {51 -- 180},
    year = {2024},
    doi = {10.4007/annals.2024.199.1.2},
    URL = {https://doi.org/10.4007/annals.2024.199.1.2}
}

@article{elmanto-sosnilo-weight,
    author = {Elmanto, Elden and Sosnilo, Vladimir},
    title = "{On Nilpotent Extensions of $\infty$-Categories and the Cyclotomic Trace}",
    journal = {International Mathematics Research Notices},
    volume = {2022},
    number = {21},
    pages = {16569-16633},
    year = {2021},
    month = {07},
    issn = {1073-7928},
    doi = {10.1093/imrn/rnab179},
    url = {https://doi.org/10.1093/imrn/rnab179},
    eprint = {https://academic.oup.com/imrn/article-pdf/2022/21/16569/46695431/rnab179.pdf},
}

@article{popescu-approximation, 
    title={General Néron desingularization and approximation}, 
    volume={104}, 
    DOI={10.1017/S0027763000022698}, 
    journal={Nagoya Mathematical Journal}, 
    author={Popescu, Dorin}, 
    year={1986}, 
    pages={85–115}
}

@article {artin-approximation,
    AUTHOR = {Artin, Michael},
     TITLE = {Algebraic approximation of structures over complete local
              rings},
   JOURNAL = {Inst. Hautes \'{E}tudes Sci. Publ. Math.},
  FJOURNAL = {Institut des Hautes \'{E}tudes Scientifiques. Publications
              Math\'{e}matiques},
    NUMBER = {36},
      YEAR = {1969},
     PAGES = {23--58},
      ISSN = {0073-8301,1618-1913},
   MRCLASS = {14.05},
  MRNUMBER = {268188},
MRREVIEWER = {H.\ Kurke},
       URL = {http://www.numdam.org/item?id=PMIHES_1969__36__23_0},
}

@article{bhatt-mathew-syntomic,
    title={Syntomic complexes and $p$-adic étale Tate twists}, volume={11}, DOI={10.1017/fmp.2022.21}, 
    journal={Forum of Mathematics, Pi}, 
    author={Bhatt, Bhargav and Mathew, Akhil}, 
    year={2023}, 
    pages={e1}
}

@article{bouis-smooth,
    url = {https://doi.org/10.1515/crelle-2023-0074},
    title = {Cartier smoothness in prismatic cohomology},
    author = {Tess Vincent Bouis},
    pages = {241--282},
    volume = {2023},
    number = {805},
    journal = {Journal für die reine und angewandte Mathematik (Crelles Journal)},
    doi = {doi:10.1515/crelle-2023-0074},
    year = {2023}
}

@online{gabber-talk,
    url = {https://www.math.princeton.edu/events/artin-grothendieck-vanishing-rigid-analytic-situations-2024-03-25t193000},
    title = {On Artin--Grothendieck vanishing in rigid-analytic situations},
    author = {Ofer Gabber},
    year = {2024}
}

@online{gabber-zavyalov,
      title={Rigid-Analytic Artin--Grothendieck Vanishing}, 
      author={Ofer Gabber and Bogdan Zavyalov},
      year={2024},
      eprint={2402.08741},
      archivePrefix={arXiv},
      primaryClass={math.AG}
}

@online{mao-crystalline,
      title={Revisiting derived crystalline cohomology}, 
      author={Zhouhang Mao},
      year={2021},
      eprint={2107.02921},
      archivePrefix={arXiv},
      primaryClass={math.AG}
}

@online{complex,
      shorthand={Complex},
      title={Condensed Mathematics and Complex Geometry}, 
      author={Dustin Clausen and Peter Scholze},
      year={2026},
      eprint={2605.11731},
      archivePrefix={arXiv},
      primaryClass={math.CV},
      url={https://arxiv.org/abs/2605.11731}, 
}

@incollection {quillen-homology,
    AUTHOR = {Quillen, Daniel},
     TITLE = {On the (co-)homology of commutative rings},
 BOOKTITLE = {Applications of {C}ategorical {A}lgebra ({P}roc. {S}ympos.
              {P}ure {M}ath., {V}ol. {XVII}, {N}ew {Y}ork, 1968)},
    SERIES = {Proc. Sympos. Pure Math.},
    VOLUME = {XVII},
     PAGES = {65--87},
 PUBLISHER = {Amer. Math. Soc., Providence, RI},
      YEAR = {1970},
   MRCLASS = {13.90 (18.00)},
  MRNUMBER = {257068},
MRREVIEWER = {S.\ Yuan},
}

@article{avramov-lci,
 ISSN = {0003486X},
 URL = {http://www.jstor.org/stable/121087},
 author = {Luchezar L. Avramov},
 journal = {Annals of Mathematics},
 number = {2},
 pages = {455--487},
 publisher = {Annals of Mathematics},
 title = {Locally complete intersection homomorphisms and a conjecture of Quillen on the vanishing of cotangent homology},
 volume = {150},
 year = {1999}
}

@article {briggs-iyengar-cotangent,
    AUTHOR = {Briggs, Benjamin and Iyengar, Srikanth B.},
     TITLE = {Rigidity properties of the cotangent complex},
   JOURNAL = {J. Amer. Math. Soc.},
  FJOURNAL = {Journal of the American Mathematical Society},
    VOLUME = {36},
      YEAR = {2023},
    NUMBER = {1},
     PAGES = {291--310},
      ISSN = {0894-0347,1088-6834},
   MRCLASS = {13D03 (13B10 14A15 14A30)},
  MRNUMBER = {4495843},
MRREVIEWER = {Haohao\ Wang},
       DOI = {10.1090/jams/1000},
       URL = {https://doi.org/10.1090/jams/1000},
}

@article {conrad-rigid,
    AUTHOR = {Conrad, Brian},
     TITLE = {Irreducible components of rigid spaces},
   JOURNAL = {Ann. Inst. Fourier (Grenoble)},
  FJOURNAL = {Universit\'{e} de Grenoble. Annales de l'Institut Fourier},
    VOLUME = {49},
      YEAR = {1999},
    NUMBER = {2},
     PAGES = {473--541},
      ISSN = {0373-0956,1777-5310},
   MRCLASS = {14G22 (32C18 32P05)},
  MRNUMBER = {1697371},
MRREVIEWER = {Lorenzo\ Ramero},
       URL = {http://www.numdam.org/item?id=AIF_1999__49_2_473_0},
}

@book {non-archimedean-analysis,
    AUTHOR = {Bosch, S. and G\"{u}ntzer, U. and Remmert, R.},
     TITLE = {Non-Archimedean analysis},
    SERIES = {Grundlehren der mathematischen Wissenschaften [Fundamental
              Principles of Mathematical Sciences]},
    VOLUME = {261},
      NOTE = {A systematic approach to rigid analytic geometry},
 PUBLISHER = {Springer-Verlag, Berlin},
      YEAR = {1984},
     PAGES = {xii+436},
      ISBN = {3-540-12546-9},
   MRCLASS = {32K10 (30G05 46P05)},
  MRNUMBER = {746961},
MRREVIEWER = {W.\ Bartenwerfer},
       DOI = {10.1007/978-3-642-52229-1},
       URL = {https://doi.org/10.1007/978-3-642-52229-1},
}

\end{document}